\RequirePackage{fix-cm}
\documentclass[smallcondensed]{svjour3}     % onecolumn (ditto)
\smartqed  % flush right qed marks, e.g. at end of proof
\usepackage{graphicx}

\usepackage[utf8]{inputenc}
\usepackage{amsmath}
\usepackage{mathrsfs}
\usepackage{algorithm,algpseudocode}
\usepackage{subfigure}
\usepackage{enumitem}
\usepackage{wasysym}%
\usepackage{hyperref}
\usepackage[makeroom]{cancel}
\usepackage{xcolor}
\newcommand\norm[1]{\left\lVert#1\right\rVert}
\newcommand{\bunderline}[1]{\underline{#1\mkern-4mu}\mkern4mu }
\newtheorem{Definition}{\textit{Definition}}
\newtheorem{Theorem}{\textit{Theorem}}

\newtheorem{Remark}{\textit{Remark}}
\renewenvironment{proof}{\textbf{\textit{Proof.}}}{\qed}
\hypersetup{
    colorlinks,
    linkcolor={blue!50!black},
    citecolor={red!50!black},
    urlcolor={blue!80!black}
}

\journalname{This is a pre-print of an article published in the Journal of Scientific Computing}

\begin{document}

\title{Truncation Error Estimation in the p-Anisotropic Discontinuous Galerkin Spectral Element Method}

\titlerunning{Truncation Error Estimation in the p-Anisotropic DGSEM}        % if too long for running head

\author{Andrés M. Rueda-Ramírez \and
        Gonzalo Rubio \and
        Esteban Ferrer \and
        Eusebio Valero
}

%\authorrunning{A. M. Rueda-Ramírez \and G. Rubio \and E. Ferrer \and E. Valero} % if too long for running head

\institute{   A. M. Rueda-Ramírez \and G. Rubio \and E. Ferrer \and E. Valero \at
              ETSIAE (School of Aeronautics), Universidad Politécnica de Madrid. \\
              Plaza de Cardenal Cisneros 3, 28040 Madrid, Spain. \\
              Tel.: +34 913 3663 26 - Ext. 205\\
              \email{am.rueda@upm.es}           %  \\
           \and
              A. M. Rueda-Ramírez \and G. Rubio \and E. Ferrer \and E. Valero
              \at
              Center for Computational Simulation, Universidad Politécnica de Madrid. \\
              Campus de Montegancedo, Boadilla del Monte, 28660 Madrid, Spain.
}

\date{Received: date / Accepted: date}

\maketitle

\begin{abstract}
In the context of Discontinuous Galerkin Spectral Element Methods (DGSEM), $\tau$-estimation has been successfully used for p-adaptation algorithms. This method estimates the truncation error of representations with different polynomial orders using the solution on a reference mesh of relatively high order.

In this paper, we present a novel anisotropic truncation error estimator derived from the $\tau$-estimation procedure for DGSEM. We exploit the tensor product basis properties of the numerical solution to design a method where the total truncation error is calculated as a sum of its directional components. We show that the new error estimator is cheaper to evaluate than previous implementations of the $\tau$-estimation procedure and that it obtains more accurate extrapolations of the truncation error for representations of a higher order than the reference mesh. The robustness of the method allows performing the p-adaptation strategy with coarser reference solutions, thus further reducing the computational cost. The proposed estimator is validated using the method of manufactured solutions in a test case for the compressible Navier-Stokes equations. 


\end{abstract}

\section{Introduction}

High-order Discontinuous Galerkin (DG) methods are becoming a popular alternative to low order methods for solving Partial Differential Equations (PDEs) because of their high accuracy and flexibility \cite{Wang2013High,Cockburn2000}. Among those, the Discontinuous Galerkin Spectral Element Method (DGSEM) \cite{kopriva2009implementing,Kopriva2002} is a nodal (collocation) version of the DG method on hexahedral meshes which allows p-anisotropic representations and has been used in a wide range of applications \cite{Kopriva2002,Rasetarinera2001,Acosta2011,Deng2007,Fraysse2016}. In the DG approach, the continuity constraint on element interfaces is relaxed, allowing for discontinuities in the numerical solution. This feature makes them more robust than continuous methods for describing advection-dominated problems, like the ones usually encountered in fluid dynamics. Moreover, DG methods can handle non-conforming meshes with hanging nodes and/or different polynomial orders efficiently, as is necessary for mesh adaptation strategies \cite{Riviere2008,Kopriva2002,Ferrer2012}.\\

Error estimates are a powerful tool in computational sciences as they quantify how accurately a numerical solution satisfies the governing mathematical equations \cite{oberkampf2010verification,roache1998verification,Roy2010}. A precise assessment of the numerical errors is useful for defect correction (a technique that enables high accuracy by correcting the numerical solution using an estimation of the error \cite{roache1998verification,Martin1996}), or for guiding mesh adaptation strategies \cite{Zienkiewicz2006,Mavriplis1989,Mavriplis1994}. The former requires highly accurate estimates of the discretization error and, therefore, a significant amount of computational resources is usually invested in computing them \cite{Phillips2011}. The latter has been broadly studied in the literature. In particular, the most common approaches are the \textit{adjoint-based} adaptation \cite{Estep1995,Hartmann2002,Hartmann2006,Venditti2002,Pierce2004}, where the numerical error of a functional (e.g. lift or drag) is estimated, which involves a high computational cost; the \textit{feature-based} adaptation, which relies on on easy-to-compute adaptation criteria, such as the assessment of jumps across element interfaces in the case of DG discretizations \cite{Remacle2003}, or the identication of large gradients \cite{Aftosmis1994,Persson2006}; and the \textit{local-error-based} adaptation \cite{Mavriplis1989,Mavriplis1994,berger1987adaptive,Fraysse2012,Syrakos2012,Syrakos2006,Kompenhans2016a,Kompenhans2016}, which depends on the assessment of any measurable local error in all the cells of the domain. A detailed comparison of the different approaches for error estimation and adaptation can be found in \cite{Fraysse2012} in the context of finite volumes or \cite{Kompenhans2016a} for high-order DG schemes. The \textit{local-error-based} adaptation methods are interesting since, in contrast to \textit{feature-based} methods, they provide a way to predict and control the overall accuracy, and are computationally cheaper than \textit{adjoint-based} schemes \cite{Kompenhans2016a,Kompenhans2016}. The topic of our work is the development of an accurate and cheap local error estimator to drive p-anisotropic adaptation in the DGSEM. \\

Two different errors are particularly relevant. On the one hand, the discretization error is the most important, but also the most difficult error to estimate \cite{Phillips2011}. It is defined as the difference between the exact and numerical solutions to the problem and can be approximated by means of solving the Discretization Error Transport Equation (DETE) \cite{Acosta2011}, an auxiliary PDE whose approximation involves the investment of further computational resources. Some of the first works using estimations of the local discretization error in high-order methods were proposed by Mavriplis \cite{Mavriplis1989,Mavriplis1994}, who developed hp-adaptation techniques for the Spectral Element Method, and Casoni et al. \cite{EvaCasoniyAntonioHuerta2011}, who used a similar approach to evaluate where to add artificial viscosity for shock capturing in Discontinuous Galerkin discretizations.\\

On the other hand, the truncation error is defined as the difference between the discrete partial differential operator and the exact partial differential operator,
\begin{equation}
\tau(\cdot) = \mathcal{R}^N(\cdot)-\mathcal{R}(\cdot),
\end{equation}
and is usually evaluated for the exact solution of the PDE \cite{Phillips2011,Rubio2015,Kompenhans2016,Kompenhans2016a}. The truncation error is related to the discretization error through the DETE \cite{Roy2010}, where it acts as a local source term. This relation makes it useful as an indicator for mesh adaptation methods \cite{Syrakos2012,Choudhary2013} since refining the mesh where the truncation error is high reduces the discretization error in all the mesh \cite{Rubio2015}, with an additional advantage: the truncation error estimation requires less computational effort. Furthermore, in hyperbolic problems the discretization error is strongly advected, i.e. it is transmitted downstream from under-resolution areas, but the truncation error is only weakly advected. Therefore, an adaptation procedure based on the truncation error targets specifically the under-resolved areas, whereas one based on the discretization error targets the under-resolved areas and the zones downstream of them \cite{rubio2015truncation,Rubio2015}. This makes the truncation error more suitable for adaptation purposes than the discretization error. Finally, it has been shown that controlling the truncation error targets the numerical accuracy of all functionals at once \cite{Kompenhans2016}, ensuring that adapting a mesh using the truncation error leads necessarily to an error decrease in any other functional (e.g. lift or drag). For all these reasons, we focus on truncation error estimators in this paper.\\

From a practical point of view, the truncation error can be estimated using a hierarchy of meshes. On the one hand, Venditti and Darmofal \cite{Venditti2002} and Phillips et al. \cite{Phillips2017,Phillips2013} studied the possibility of estimating the truncation error by evaluating a coarse grid solution in the partial differential operator of a fine grid, an approach known as the coarse-to-fine approach. On the other hand, the fine-to-coarse approach, also known as the $\tau$-estimation method, was introduced by Brandt \cite{Brandt1984} and consists in estimating the local truncation error by using a fine grid solution interpolated to the coarse grid. Phillips \cite{Phillips2014} showed that the fine-to-coarse ($\tau$-estimation) method produces more accurate results than the coarse-to-fine approach and, therefore, it is the one retained in this work.\\

The $\tau$-estimation approach has been successfully used for adaptation purposes in low-order Finite Difference \cite{berger1987adaptive} and Finite Volume schemes \cite{Fraysse2012,Syrakos2012,Syrakos2006}. Moreover, Rubio et al. extended it to high-order methods using a continuous Chebyshev collocation method \cite{Rubio2013} and later the Discontinuous Galerkin Spectral Element Method (DGSEM) \cite{Rubio2015}. In that work, they studied the \textit{quasi-a priori} truncation error estimation, which allows estimating the truncation error without having fully converged fine solutions, and introduced the concept of \textit{isolated} truncation error (valid only for DG formulations), which only considers inner elemental contributions to the error and neglects the upwind contributions. More recently, Kompenhans et al. \cite{Kompenhans2016} applied these estimators to perform p-anisotropic adaptation for the Euler and Navier-Stokes equations, and compared $\tau$-based to featured based adaption, showing better performance for the former \cite{Kompenhans2016a}. The adaptation strategy consisted in converging a high order representation (reference mesh) to a specified global residual and then performing a single error estimation followed by a corresponding p-adaptation process. Even though their methodology is very promising, we will show that it produces a large underestimation of the error for polynomial orders that are higher than the ones in the original reference mesh. This fact makes necessary to compute the initial solution in a very refined reference mesh to avoid inaccuracies. \\

In this paper, we extend the work on high-order $\tau$-estimators by Rubio et al. \cite{Rubio2013,Rubio2015,rubio2015truncation}, and formulate a new anisotropic truncation error estimator that exploits the tensor product basis expansion of the DGSEM. The new error estimator is shown to be suitable for performing anisotropic p-adaptation, and to have two main advantages over existing truncation error estimators; first, that it requires fewer operations to estimate the truncation error of all possible combinations of polynomial orders; and second, that it yields more accurate estimations of the truncation error for representations of a higher order than the reference mesh. This feature allows using reference meshes of a lower polynomial order, hence reducing the computational cost. We also analyze the properties of the traditional \textit{non-isolated} truncation error and the \textit{isolated} truncation error. To the authors' knowledge, this is the first time that a high-order truncation error estimator based on the $\tau$-estimation method is formulated in an anisotropic/decoupled way, analyzed and tested. \\

The paper is organized as follows. In section \ref{sec:Maths}, we present the mathematical background. First, the Discontinuous Galerkin Spectral Element Method is briefly summarized; then, we detail the existing techniques for approximating the truncation error of isotropic and anisotropic representations using the $\tau$-estimation method. In section \ref{sec:NewTechnique}, the proposed anisotropic $\tau$-estimator is introduced and analyzed. In section \ref{sec:Validation}, we present a validation of the assumptions needed for formulating the new approach and study the properties of the proposed method by means of a manufactured solutions test case of the compressible Navier-Stokes equations. Finally, the conclusions are summarized in section \ref{sec:Conclusions}.

\section{Mathematical background} \label{sec:Maths}
In section \ref{sec:DGSEM}, we describe briefly the DGSEM approach. Section \ref{sec:ErrorDefinitions} contains the error definitions that will be used throughout the paper and provides an insight into the convergence properties of the different error measures. In section \ref{sec:TauEstim}, we review the $\tau$-estimation method for DGSEM schemes, and then we explain in section \ref{sec:OldPolOrders} how it has been used in the literature for obtaining anisotropic error extrapolations. 

\subsection{The Discontinuous Galerkin Spectral Element Method (DGSEM)}\label{sec:DGSEM}
We consider the approximation of systems of conservation laws, 

\begin{equation}\label{eq:NScons}
\mathbf{q}_t + \nabla \cdot \mathscr{F} = \mathbf{s},
\end{equation}
where $\mathbf{q}$ is the vector of conserved variables, $\mathscr{F}$ is the flux dyadic tensor which depends on $\mathbf{q}$, and $\mathbf{s}$ is a source term. This system represents, among others, the compressible Navier-Stokes equations, as detailed in Appendix \ref{sec:NS}. Multiplying equation \ref{eq:NScons} by a test function $\mathbf{v}$ and integrating by parts over the domain $\Omega$ yields the weak formulation:
\begin{equation} \label{eq:weak}
\int _{\Omega} \mathbf{q}_t \mathbf{v} d \Omega - \int _{\Omega} \mathscr{F} \cdot \nabla \mathbf{v} d \Omega + \int_{\partial \Omega} \mathscr{F} \cdot \mathbf{n} \mathbf{v} d \sigma = \int _{\Omega} \mathbf{s} \mathbf{v} d \Omega,
\end{equation}
where $\mathbf{n}$ is the normal unit vector on the boundary $\partial \Omega$. Let the domain $\Omega$ be approximated by a tessellation $\mathscr{T} = \lbrace e \rbrace$, a combination of $K$ finite elements $e$ of domain $\Omega^e$ and boundary $\partial \Omega^e$. Moreover, let $\mathbf{q}$, $\mathbf{s}$, $\mathscr{F}$ and $\mathbf{v}$ be approximated by piece-wise polynomial functions (that are continuous in each element) defined in the space of $L^2$ functions

\begin{equation}
\mathscr{V}^N = \lbrace \mathbf{v}^N \in L^2(\Omega) : \mathbf{v}^N\vert_{\Omega^e} \in \mathscr{P}^N(\Omega^e) \ \ \forall \ \Omega^e \in \mathscr{T} \rbrace,
\end{equation}
where $\mathscr{P}^N(\Omega^e)$ is the space of polynomials of degree at most $N$ defined in the domain of the element $e$. Remark that the functions in $\mathscr{V}^N$ may be discontinuous at element interfaces and that the polynomial order $N$ may be different from element to element. Equation \ref{eq:weak} can then be rewritten for each element as:

\begin{multline} \label{eq:weak2}
\int _{\Omega^e} {\mathbf{q}^e_t}^N {\mathbf{v}^e}^N d \Omega^e - \int_{\Omega^e} {\mathscr{F}^e}^N \cdot \nabla {\mathbf{v}^e}^N d \Omega^e \\+ \int_{\partial \Omega^e} {\mathscr{F}^*} \left( {\mathbf{q}^e}^N, {\mathbf{q}^-}^N, \mathbf{n} \right)  {\mathbf{v}^e}^N d \sigma^e = \int _{\Omega^e} {\mathbf{s}^e}^N {\mathbf{v}^e}^N d \Omega^e,
\end{multline}
where the superindex ``$e$" refers to the functions as evaluated inside the element $e$, i.e. ${\mathbf{q}^e}^N = {\mathbf{q}^N}\rvert_{\Omega^e}$; whereas the superindex ``$-$" refers to the value of the functions on the external side of the interface $\partial \Omega^e$. The numerical flux function, ${\mathscr{F}^*}$, allows to uniquely define the flux at the element interfaces and to weakly prescribe the boundary data as a function of the conserved variable on both sides of the boundary/interface (${\mathbf{q}^{e}}^N$ and ${\mathbf{q}^{-}}^N$) and the normal vector ($\mathbf{n}$). Multiple choices for the numerical flux functions can be found in the literature \cite{toro2013riemann}. In the present work, we use Roe \cite{Roe1981} as the advective Riemann Solver and Bassi-Rebay 1 \cite{Bassi1997} as the diffusive Riemann solver. Remark that the numerical flux must be computed in a specific manner when the representation is non-conforming \cite{Kopriva2002}. \\

Since $\mathbf{q}^N$, $\mathbf{s}^N$, $\mathbf{v}^N$ and $\mathscr{F}^N$ belong to the polynomial space $\mathscr{V}^N$, it is possible to express them inside every element as a linear combination of basis functions $\phi_n \in \mathscr{P}^N(\Omega^e)$,

\begin{eqnarray}\label{eq:PolExp}
\mathbf{q}\rvert_{\Omega^e} \approx {\mathbf{q}^e}^N = \sum_n \mathbf{Q}^e_n \phi^e_n (\mathbf{x}), & & \ \ \ \ 
\mathbf{s}\rvert_{\Omega^e} \approx {\mathbf{s}^e}^N = \sum_n \mathbf{S}^e_n \phi^e_n (\mathbf{x}), \nonumber \\
\mathbf{v}\rvert_{\Omega^e} \approx {\mathbf{v}^e}^N = \sum_n \mathbf{V}^e_n \phi^e_n (\mathbf{x}), & &
\mathscr{F}\rvert_{\Omega^e} \approx {\mathscr{F}^e}^N = \sum_n \pmb{\mathscr{F}}^e_n \phi^e_n (\mathbf{x}).
\end{eqnarray}

Therefore, equation \ref{eq:weak2} can be expressed in a discrete form as

\begin{equation} \label{eq:DiscretNSElem}
[\mathbf{M}]^e \frac{\partial \mathbf{Q}^e}{\partial t} + \mathbf{F}^e(\mathbf{Q}) = [\mathbf{M}]^e \mathbf{S}^e,
\end{equation}
where $\mathbf{Q}^e=(\mathbf{Q}^e_1, \mathbf{Q}^e_2, \cdots, \mathbf{Q}^e_n, \cdots)^T$ is the local solution that contains the coefficients of the linear combination for the element $e$; $\mathbf{Q}=(\mathbf{Q}^1,\mathbf{Q}^2, \cdots, \mathbf{Q}^K)^T$ is the global solution that contains the information of all elements; $[\mathbf{M}]^e$ is known as the elemental mass matrix, and $\mathbf{F}^e(\cdot)$ is a nonlinear spatial discrete operator on the element level:
\begin{align}
[\mathbf{M}]^e_{i,j} &= \int_{\Omega^e} \phi^e_i \phi^e_j  d \Omega^e \\
\mathbf{F}^e(\mathbf{Q})_j &= \sum_i \left[ - \int_{\Omega^e} \pmb{\mathscr{F}}_i^e \cdot \phi^e_i \nabla \phi^e_j d \Omega^e \right] + \int_{\partial \Omega^e} {\mathscr{F}^*}^N \left( \mathbf{Q}^e, \mathbf{Q}^{-}, \mathbf{n} \right)  \phi^e_j d \sigma^e.
\end{align}

Note that the operator $\mathbf{F}^e$ is applied on the global solution, since it is the responsible for connecting the elements of the mesh (weakly). Assembling the contributions of all elements into the global system we obtain:

\begin{equation} \label{eq:DiscretNS}
[\mathbf{M}] \frac{\partial \mathbf{Q}}{\partial t} + \mathbf{F}(\mathbf{Q}) = [\mathbf{M}] \mathbf{S}.
\end{equation}

In the DGSEM \cite{kopriva2009implementing}, the tesselation is performed with non-overlapping hexahedral elements of order $N=(N_1,N_2,N_3)$ (independent in every direction) and the integrals are evaluated numerically by means of a Gaussian quadrature that is also of order $N=(N_1,N_2,N_3)$. For complex geometries, it is most convenient to perform the numerical integration in a reference element and transform the results to the physical space by means of a high-order mapping:

\begin{eqnarray}
\mathbf{x}^e = \mathbf{x}^e \left( \pmb{\xi} \right), & & \pmb{\xi} = \left( \xi, \eta, \zeta \right) \in \left[ -1, 1 \right]^3,
\end{eqnarray}
where the order of ${x}^e_i$ is at most $N_i^e$ (subparametric or at most isoparametric mapping). The differential operators can be expressed in the reference element in terms of the covariant ($\mathbf{a}_i$) and contravariant ($\mathbf{a}^i$) metric tensors:

\begin{equation}
\mathbf{a}_i = \frac{\partial \mathbf{x}^e}{\partial \xi_i}, \ \ \mathbf{a}^i = \nabla \xi_i, \ \ i = 1, 2, 3.
\end{equation}

Under these mappings, the gradient and divergence operators become:

\begin{equation}
\nabla q = \frac{1}{J} \sum_{i=1}^d \frac{\partial}{\partial \xi_i} \left( J\mathbf{a}^i q \right), \ \ \nabla \cdot \mathbf{f} = \frac{1}{J} \sum_{i=1}^d \frac{\partial}{\partial \xi_i} \left( J\mathbf{a}^i \cdot \mathbf{f} \right),
\end{equation}
where the Jacobian of the transformation can be expressed in terms of the covariant metric tensor:

\begin{equation}
J = \mathbf{a}_i \cdot \left( \mathbf{a}_j \times \mathbf{a}_k \right), \ \ \left( i,j,k \right) \ cyclic.
\end{equation}

For details on how to compute the metric terms for 2D and 3D geometries, see \cite{Kopriva2006}. Furthermore, in the DGSEM the polynomial basis functions ($\phi_n$ in equation \ref{eq:PolExp}) are tensor product reconstructions of Lagrange interpolating polynomials on quadrature points in each of the Cartesian directions of the reference element:

\begin{equation}
\mathbf{q}^N = \sum_n \mathbf{Q}_n \phi_n (\mathbf{x}) = \sum_{i=0}^{N_1} \sum_{j=0}^{N_2} \sum_{k=0}^{N_3} \mathbf{Q}_{i,j,k} l_i (\xi) l_j (\eta) l_k (\zeta).
\end{equation}

Therefore, $\mathbf{Q}_n=\mathbf{Q}_{i,j,k}$ are simply the nodal values of the solution, and $[\mathbf{M}]$ is a diagonal matrix containing the quadrature weights and the mapping terms. In the present work, we make use of the Legendre-Gauss quadrature points \cite{kopriva2009implementing}.

\subsection{Definition of Errors} \label{sec:ErrorDefinitions}
In this section, we define some measures of the error that will be used throughout the paper.

\begin{Definition}[Interpolation error]
The difference between a function and its polynomial interpolant of order $N$: 
\end{Definition}

\begin{equation} \label{eq:InterpError}
\varepsilon^N_{\mathbf{q}}= \mathbf{q} - \mathbf{I}^N \mathbf{q},
\end{equation}
where $\mathbf{I}^N\mathbf{q}$ is the function that can be reconstructed using the polynomial expansion with coefficients $\mathbf{Q}_i$ (equation \ref{eq:PolExp}). For sufficiently smooth functions, in the asymptotic range the interpolation error in an element $e$ behaves as: 

\begin{equation} \label{eq:InterpErrorConvergence}
\norm{ \varepsilon^N_{\mathbf{q}}  \bigr\rvert_{\Omega^e} } \le C^e_0 e^{-N^e \eta^e_0},
\end{equation}
where $C^e_0$ and $\eta^e_0$ are constants that depend on the local smoothness of the function $\mathbf{q}$ \cite{canuto2012spectral,hesthaven2007nodal} and $N^e$ is the local polynomial order in the element $e$. In the DGSEM, the use of tensor-product bases in $d$ dimensions allows decoupling the interpolation error in directional components, each of which depends solely on the polynomial order in the corresponding direction:

\begin{equation} \label{eq:DirectionalInterpolError}
\varepsilon_{\mathbf{q}}^N = \sum_{i=1}^d \varepsilon_{\mathbf{q},i}^{N} \ \ \textrm{such that} \ \ \varepsilon_i = \varepsilon_i (N_i).
\end{equation}

As a consequence, in the p-anisotropic DGSEM, the interpolation error exhibits a tensor-product-type error bound in $d$ dimensions inside every element \cite{Rubio2015},

\begin{equation} \label{eq:AnisInterpErrorConvergence}
\norm{ \varepsilon_{\mathbf{q}}^N \big\rvert_{\Omega^e} } \le \sum_{i=1}^d C^e_{0,i} e^{-N^e_i \eta^e_{0,i}}.
\end{equation}

%%%%%%%%%%%%%%%%%%%%%%%%%%%%%%%%%%%%%%%%%%%%%%%%%%%%%%%%%%%%%%%%%%%%%%%%%

\begin{Definition}[Discretization error] \label{Def:DiscError}
The difference between the exact solution to the problem, $\bar{\mathbf{q}}$, and the one obtained with a discretization of order $N$, $\bar{\mathbf{q}}^N$:
\begin{equation} \label{eq:DiscError}
\epsilon^N= \bar{\mathbf{q}} - \bar{\mathbf{q}}^N.
\end{equation}
\end{Definition}

The discretization error in an element is influenced by other elements because of the advection properties of the PDE. In fact, we will decouple the discretization error in  \textit{locally-generated} and \textit{externally-generated} contributions for every element:

\begin{equation} \label{eq:DiscErrorLocalExt}
\epsilon^N \bigr\rvert_{\Omega^e} = \epsilon^N_{\Omega^e} + \epsilon^N_{\partial\Omega^e}.
\end{equation}

In the p-isotropic DGSEM, it can be assumed that the discretization error in each element behaves as \cite{Rubio2015,rubio2015truncation}:

\begin{equation} \label{eq:DiscErrorConvergence} 
\norm{\epsilon^N \bigr\rvert_{\Omega^e}} \le C^e_{\epsilon} e^{-N^e \eta^e_{\epsilon}} + \sum_{\substack{k=1 \\ k \ne e }}^K C^{*k}_{\epsilon} e^{-N^k \eta^{*k}_{\epsilon}},
\end{equation} 
where $K$ is the number of elements, $C^e_{\epsilon}$ and $\eta^e_{\epsilon}$ are constants that depend on the smoothness of the solution in the element $e$ \cite{canuto2012spectral,hesthaven2007nodal}, and $C^{*k}_{\epsilon}$ and $\eta^{*k}_{\epsilon}$ are constants that depend both on the smoothness of the solution and the advection properties of the PDE. The first term on the right-hand side corresponds to the bound of the \textit{locally-generated} discretization error ($\epsilon^N_{\Omega}$), whereas the second term is the bound of the \textit{externally-generated} discretization error ($\epsilon^N_{\partial \Omega}$) which gathers the errors that are introduced through the Riemann solver. Note that $\varepsilon^N$ is the minimum possible (lower bound of) $\epsilon^N$. For an anisotropic representation in $d$ dimensions, the expression becomes \cite{Rubio2015,rubio2015truncation,Georgoulis2003}:

 \begin{equation} \label{eq:AnisDiscErrorConvergence}
\norm{ \epsilon^N \bigr\rvert_{\Omega^e} } \le \sum_{i=1}^d C^e_{\epsilon,i} e^{-N^e_{i} \eta^e_{\epsilon,i}} + \sum_{\substack{k=1 \\ k \ne e }}^K \sum_{i=1}^d C^k_{\epsilon,i} e^{-N^k_{i} \eta^k_{\epsilon,i}}.
\end{equation}

%%%%%%%%%%%%%%%%%%%%%%%%%%%%%%%%%%%%%%%%%%%%%%%%%%%%%%%%%%%%%%%%%%%%%%%%

\begin{Definition}[Quadrature error]\label{Def:QuadError}
The quadrature error, also referred to as the numerical integration error, is the difference between the exact integral of a function and its approximation by a Gaussian quadrature:
\end{Definition}
\begin{equation}
e_{\int}^N = \int_{\Omega} q d\Omega - \int^N_{\Omega} q d\Omega,
\end{equation}
where the superindex $N$ on the integral indicates that it is approximated using a Gaussian quadrature of order $N$,
\begin{equation}
\int^N_{\Omega} q d\Omega = \sum_{j=0}^N q_j w_j,
\end{equation}
and $w_j$ are the quadrature weights.

%%%%%%%%%%%%%%%%%%%%%%%%%%%%%%%%%%%%%%%%%%%%%%%%%%%%%%%%%%%%%%%%%%%%%%%%

\begin{Definition}[\textit{Non-isolated} truncation error]\label{Def:TruncError}
We define the \textit{non-isolated} truncation error of a discretization of order $N$ as the difference between the discrete partial differential operator of order $N$ and the exact partial differential operator applied to the exact solution:
\begin{equation}\label{eq:TruncError}
\tau^N = \mathcal{R}^N(\mathbf{I}^N \bar{\mathbf{q}})-\mathcal{R}(\bar{\mathbf{q}}).
\end{equation}
\end{Definition}

The exact partial differential operator can be derived from equation \ref{eq:NScons} as 

\begin{equation}
\mathcal{R}(\mathbf{q}) = \mathbf{s} - \nabla \cdot \mathscr{F} = \mathbf{q}_t, 
\end{equation}
and the discrete partial differential operator is derived from equation \ref{eq:DiscretNS} as 

\begin{equation} \label{eq:Rdiscrete} 
\pmb{\mathcal{R}}^N(\pmb{I}^N\mathbf{q}) = [\mathbf{M}] \mathbf{S} - \mathbf{F}(\pmb{I}^N\mathbf{q}),
\end{equation}
where $\pmb{\mathcal{R}}^N$ contains the sampled values of $\mathcal{R}^N$ in all the nodes of the domain and $\pmb{I}^N$ is a sampling operator. $\mathcal{R}^N$ is reconstructed  from $\pmb{\mathcal{R}}^N$ element-wise with equation \ref{eq:PolExp}. Since for steady state, $\mathcal{R}(\bar{\mathbf{q}})=0$, the \textit{non-isolated} truncation error can be then computed inserting equation \ref{eq:Rdiscrete} into \ref{eq:TruncError} as 
\begin{equation}\label{eq:TruncErrorSteady}
\pmb{\tau}^N=\pmb{\mathcal{R}}^N(\pmb{I}^N\bar{\mathbf{q}})=[\mathbf{M}] \mathbf{S} - \mathbf{F}(\pmb{I}^N\bar{\mathbf{q}}).
\end{equation}

The dependence of the \textit{non-isolated} truncation error on the discretization error is obtained by using definition \ref{Def:DiscError} and expanding equation \ref{eq:TruncErrorSteady} as a Taylor series:

\begin{equation}\label{eq:TruncErrorDiscError}
\pmb{\tau}^N=\frac{\partial \pmb{\mathcal{R}}^N}{\partial \mathbf{Q}^N} \biggr\rvert_{\bar{\mathbf{Q}}^N} \pmb{\epsilon}^N + \mathcal{O}((\pmb{\epsilon}^N)^2).
\end{equation}

Taking into account equations \ref{eq:TruncErrorDiscError} and \ref{eq:AnisDiscErrorConvergence}, and based on previous numerical results \cite{rubio2015truncation,Rubio2015}, Kompenhans et al. \cite{Kompenhans2016} stated that the truncation error in an element is bounded by

\begin{equation} \label{eq:TruncErrorAnisConv}
\norm{ \tau^N \bigr\rvert_{\Omega^e} } \le \sum_{i=1}^d C^e_{i}e^{-N^e_{i} \eta^e_{i}} + \norm{ \tau_{\partial \Omega^e}}.
\end{equation}
This expression was validated experimentally \cite{Kompenhans2016,Kompenhans2016a}. The first term in equation \ref{eq:TruncErrorAnisConv} is the bound of the \textit{locally-generated} truncation error, whereas the second term is the bound of the \textit{externally-generated} truncation error that enters through the Riemann solver and does not depend on the local polynomial orders. The second term is a consequence of the dependence of the discretization error on the solution in other elements.

%%%%%%%%%%%%%%%%%%%%%%%%%%%%%%%%%%%%%%%%%%%%%%%%%%%%%%%%%%%%%%%%%%%%%%%%

\begin{Definition}[\textit{Isolated} truncation error]\label{Def:IsolTruncError}

The \textit{isolated} truncation error is defined as \cite{Rubio2015}
\end{Definition}
\begin{equation}
\hat \tau^N = \mathcal{\hat R}^N(\mathbf{I}^N \bar{\mathbf{q}}),
\end{equation}
where $\mathcal{\hat R}^N(\cdot)$ is the \textit{isolated} discrete partial differential operator, which is derived in the same manner as $\mathcal{R}^N(\cdot)$, but $\mathscr{F}$ is not substituted by $\mathscr{F}^*$ during the process (equation \ref{eq:weak2}). Therefore, the sampled form of the discrete isolated partial differential operator yields

\begin{equation}\label{eq:IsolTruncErrorSteady}
\hat{ \pmb{\tau}}^N = \hat{ \pmb{\mathcal{R}}}^N(\pmb{I}^N \bar{\mathbf{q}} )=[\mathbf{M}] \mathbf{S} - \mathbf{\hat F}(\pmb{I}^N \bar{\mathbf{q}} ),
\end{equation}
where the elemental contribution to the nonlinear discrete operator is

\begin{equation} \label{eq:IsolatedDiscreteNonlinOperator}
\mathbf{\hat F}^e(\mathbf{Q}^e)_j = \sum_i \left[ - \int^N_{\Omega^e} \pmb{\mathscr{F}}_i \cdot \phi^e_i \nabla \phi^e_j d \Omega^e \right] + \int^N_{\partial \Omega^e} \mathscr{F}^N \cdot \mathbf{n}  \phi^e_j d \sigma^e.
\end{equation}

This change eliminates the influence of the neighboring elements and boundaries in the truncation error of each element. The dependence of the \textit{isolated} truncation error on the interpolation error of the fluxes inside an element ($\varepsilon^N_{\mathscr{F},e}$) can be expressed as (see Appendix \ref{sec:IsolTruncError} and \cite{Rubio2015}):

\begin{equation} \label{eq:IsolTruncErrorInterpol}
\hat \tau^N \bigr\rvert_{\Omega^e} \approx \left( \nabla \cdot \varepsilon^N_{\mathscr{F}} \bigr\rvert_{\Omega^e}, \phi \right)^N_{\Omega^e}.
\end{equation}

This shows that $\hat \tau$ indeed depends only on the discrete representation of the numerical solution in the element $e$. Rubio et al. \cite{Rubio2015} pointed out that the \textit{isolated} truncation error might be a better sensor for adaptation algorithms for hyperbolic PDEs than its \textit{non-isolated} counterpart or the discretization error, since unlike the last two, the first one is not affected by neighbors' errors. Notice that equation \ref{eq:IsolTruncErrorInterpol} resembles the DETE \cite{Roy2010}. In this case, the isolated truncation error acts as a source term for the interpolation error. In consequence, decreasing the \textit{isolated} truncation error reduces the interpolation error.\\

Finally, it is important to note that the spectral convergence of the \textit{isolated} truncation error is similar to the \textit{non-isolated} truncation error \cite{Rubio2015,rubio2015truncation} and can be expressed as
\begin{equation} \label{eq:IsolTruncErrorAnisConv}
\norm{ \hat \tau_e^N \bigr\rvert_{\Omega^e} } \le \sum_{i=1}^d C^e_{i}e^{-N^e_{i} \eta^e_{i}}.
\end{equation}

Remark that, because of the reasons exposed above, in this case there is no \textit{externally-generated} truncation error. The \textit{hat notation} will be dropped from now on since, unless explicitly stated, the formulations in this paper are valid for both the \textit{non-isolated} and the \textit{isolated} truncation errors.\\

\subsection{$\tau$-Estimation method} \label{sec:TauEstim}

Since in general the exact solution to the problem is not available, we are interested in using an estimation for equations \ref{eq:TruncErrorSteady} and \ref{eq:IsolTruncErrorSteady}. The $\tau$-estimation method makes use of an approximate solution on a reference mesh of order $P>N$ instead of the exact one. The most straightforward methodology is to converge this high-order solution to a low residual near machine round-off, $\bar{\mathbf{Q}}^P$. This is known as the \textit{a posteriori} approach. In practice, one can also use a non-converged solution, $\mathbf{\tilde Q}^P$. This is known as the \textit{quasi-a priori} approach. In this paper, we use the following formulation, which is valid for the \textit{a-posteriori} method and the \textit{quasi a-priori} approach without correction:

\begin{equation}\label{eq:TruncError!}
\pmb{\tau}_P^N = \pmb{\mathcal{R}}^N(\mathbf{I}_P^N\mathbf{Q}^P) = [\mathbf{M}^N]\mathbf{S}^N-\mathbf{F}^N(\tilde{\mathbf{Q}}^P),
\end{equation}
where $\mathbf{I}_P^N$ is an interpolation operator from order $P$ to order $N$. For compactness, the notation of this work omits the interpolation matrix such that $\pmb{\mathcal{R}}^N(\mathbf{Q}^P) = \pmb{\mathcal{R}}^N(\mathbf{I}_P^N\mathbf{Q}^P)$. Equation \ref{eq:TruncError!} is valid for both the \textit{isolated} (inserting the $hats$) and the \textit{non-isolated} truncation error. Note that the truncation error estimation can be easily performed for anisotropic polynomial representations of $d$ dimensions. For instance, in a 2D anisotropic case, equation \ref{eq:TruncError!} can be rewritten as:

\begin{equation}\label{eq:TruncError2D!}
\pmb{\tau}_{P_1P_2}^{N_1N_2} = [\mathbf{M}^{N_1N_2}]\mathbf{S}^{N_1N_2}-\mathbf{F}^{N_1N_2}(\tilde{\mathbf{Q}}^{P_1P_2}).
\end{equation}

\subsection{Low order extrapolation of the truncation error estimations}\label{sec:OldPolOrders}

In this section, we review the method proposed by Kompenhans et al. \cite{Kompenhans2016} to extrapolate the $\tau$-estimations of anisotropic representations. This method was successfully used to perform a p-adaptation strategy \cite{Kompenhans2016,Kompenhans2016a}. We will show that their strategy can be classified as a low order extrapolation. In order to do so, let us first introduce the concept of truncation error map.

\begin{Definition} [Truncation error map] \label{Def:TruncErrorMap}
The (graphical) representation of the truncation error behavior inside an element with respect to the polynomial order as a $(d+1)$-dimensional plot of $\log \norm{\tau^N}$ as a function of the polynomial order in every direction of the reference element $N=(N_1,\cdots,N_d)$, where $d$ is the number of dimensions. 
\end{Definition}

Because of the spectral convergence of the truncation error in the asymptotic range, the one-dimensional (or isotropic $d$-dimensional) map turns out to be a discrete scatter plot of points that describe a linear function with a negative slope (the convergence rate $\eta$), as shown in Figure \ref{fig:TruncErrorMap1D}. \\

\begin{figure}[h]
\begin{center}
\subfigure[One-dimensional map]{\label{fig:TruncErrorMap1D} \includegraphics[width=0.45\textwidth]{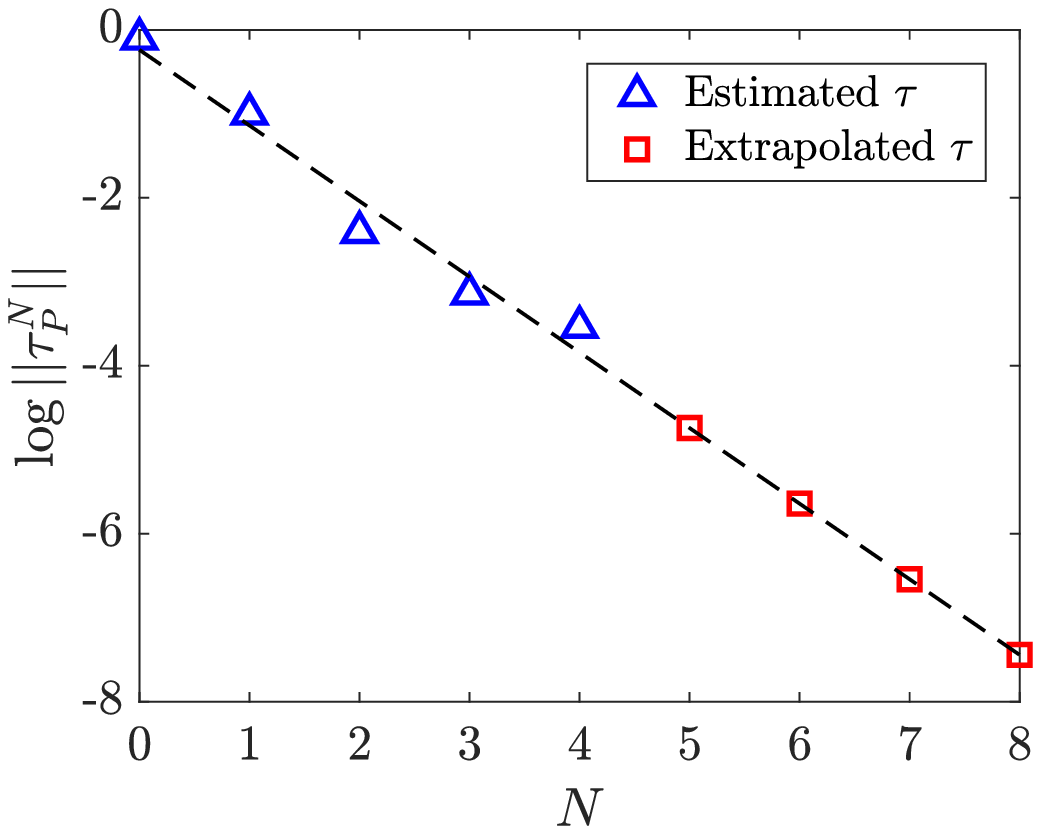}} \qquad
\subfigure[Two-dimensional map used by Kompenhans et al. \cite{Kompenhans2016}]{\label{fig:HyperPlaneKompen} \includegraphics[width=0.45\textwidth]{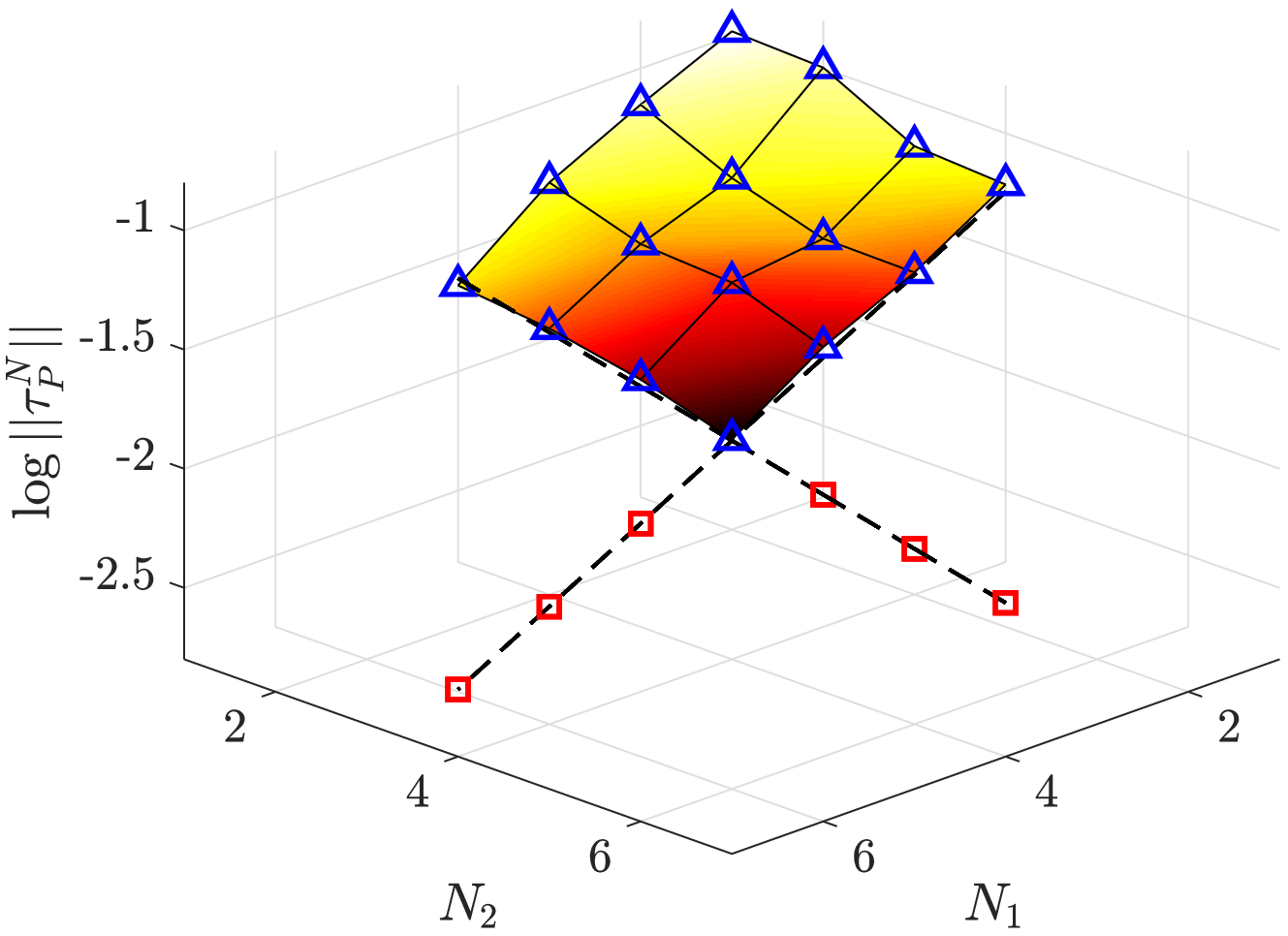}}

\caption{One- and two-dimensional truncation error maps constructed with $P=5$ showing estimated and extrapolated values for a toy problem (illustrative)}\label{fig:TruncErrorMapKompen}
\end{center}
\end{figure} 

Kompenhans et al. \cite{Kompenhans2016} used the estimated truncation error map to adapt the polynomial orders of a given mesh using a specified maximum permitted error threshold, $\tau_{max}$. The method for estimating the map consists of four steps:

\begin{enumerate}
\item Generate an \textit{inner} map for $N_i<P_i$ using equation \ref{eq:TruncError2D!}. This requires $n_{eval}$ evaluations of the discrete partial differential operator $\mathcal{R}^N$, where
\begin{equation}
n_{eval} = \prod_{i=1}^d (P_i - 1).
\end{equation}
The estimated points of the \textit{inner} map are marked as blue triangles in Figure \ref{fig:TruncErrorMapKompen}.

\item Use the \textit{inner} map to look for a combination of polynomial orders that fulfills the specified error threshold. If a combination fulfills $\tau_{max}$, adapt the polynomial order and exit the adaptation process. Otherwise, additional considerations are required.

\item Compute $\log ||\tau^{N_iN_j}||$ and perform a linear regression analysis in the direction $i$ in order to describe the behavior of $\log \norm{\tau}$ as a function of $N_i$ ($N_j=P_j-1$). The result of the linear regression is marked with a dashed line in Figure \ref{fig:TruncErrorMapKompen}.

\item Use the linear regression to estimate the truncation error for $N_i \ge P_i$, and select the value of $N_1$ and $N_2$ independently from these extrapolations. The extrapolated values of the truncation error are marked with red squares in Figure \ref{fig:TruncErrorMapKompen}. 

\end{enumerate}

This procedure is performed for every element in all the Cartesian directions to adapt the mesh. For further details, refer to the original paper by Kompenhans et al. \cite{Kompenhans2016} and to our example in section \ref{sec:ComparisonKompen}.

\subsubsection{Analysis of the method}
Two remarks can be made about the described four-step procedure:

\begin{Remark} \label{rem:AssumptionKompen}
Steps 3 and 4 assume that the spectral convergence observed in 1D extends to higher dimensions along iso-$N_i$ lines of the truncation error map.
\end{Remark}

\begin{Remark} \label{rem:TruncErrorReq}
For the \textit{non-isolated} truncation error, the behavior shown in Figure \ref{fig:TruncErrorMapKompen} can only be expected for the \textit{locally-generated} component (see equation \ref{eq:TruncErrorAnisConv}). This means that the extrapolation procedure may predict unexpected behaviors if the truncation error in neighboring elements is high, i.e., if the $\tau$-estimation procedure is not performed element-wise while keeping the polynomial order in other elements \textit{sufficiently}\footnote{Sufficiently high does not necessarily mean that the polynomial order of the other elements must be kept in $P$, but that it must be high enough so that the \textit{externally-generated} contributions to the truncation error are less than the \textit{internally-generated} ones.} high. 
\end{Remark}

As stated in the remark \ref{rem:AssumptionKompen}, the extrapolation procedure assumes that the truncation error decreases exponentially along iso-$N_i$ lines of the truncation error map. That is the same as saying that the truncation error map is a plane for $d=2$, and in general that it is a hyperplane of dimension $d+1$. In Figure \ref{fig:HyperPlane} we present an illustration that resembles the hyperplane behavior in two dimensions for perfect spectral convergence. The described methodology consists in constructing $d$ iso-$N_i$ lines on the hyperplane, which should contain the values of the truncation error for $N_j=P_j-1 \ \forall \ N_i$ (red line with triangular markers and black line with circular markers in Figure \ref{fig:HyperPlane}). In that scenario, selecting $N_i$ independently can be regarded as a conservative criterion, since in the hyperplane we have:

\begin{itemize}
\item In 2D:
\begin{align*}
\tau^{N_1N_2} &\le \tau^{P_1N_2}, \\
\tau^{N_1N_2} &\le \tau^{N_1P_2}.
\end{align*}

\item In 3D:
\begin{align*}
\tau^{N_1N_2N_3} &\le \tau^{P_1N_2N_3}, \\
\tau^{N_1N_2N_3} &\le \tau^{N_1P_2N_3}, \\
\tau^{N_1N_2N_3} &\le \tau^{N_1N_2P_3}.
\end{align*}

\end{itemize}

for $N_i \ge P_i$. See Figure \ref{fig:HyperPlane}.\\

\begin{figure}[h!]
\begin{center}
\includegraphics[width=0.45\textwidth]{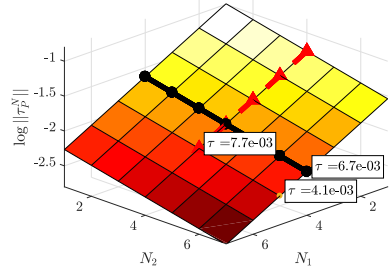}
\caption{Hyperplane behavior of the truncation error of a toy problem (illustrative)}\label{fig:HyperPlane}
\end{center}
\end{figure}

Hereinafter, the method by Kompenhans et al. shall be referred to as the \textit{low order extrapolation method}, since it supposes that the truncation error map ($\log||\tau||$) has a linear behavior. In light of the analysis in section \ref{sec:NewTechnique}, we will be able to formulate a \textit{high order extrapolation method} that provides extrapolated estimations with increased accuracy.

\section{New anisotropic truncation error estimation}\label{sec:NewTechnique}

In this section, we present the new anisotropic truncation error estimator, discuss some of its properties and compare them with the error estimators that have been used in the literature for performing anisotropic p-adaptation. The formulation of this new estimator involved a mathematical proof based on specific assumptions that is detailed in section \ref{sec:AnisTauEstimation}. In section \ref{sec:AnisConv}, we analyze the convergence behavior of the anisotropic estimator. In section \ref{sec:NewPolOrders}, we describe how the new estimator can be used for approximating the truncation error of higher-order representations, and prove that it is superior to existing $\tau$-estimators at obtaining these approximations. \\

\subsection{Anisotropic $\tau-$estimation} \label{sec:AnisTauEstimation}
The anisotropic $\tau$-estimator is a generalization of the ideas reviewed in section \ref{sec:TauEstim} and is based on four assumptions that are explained first. For the sake of readability and without loss of generality, all the mathematical statements in this section are for 2D formulations, where $N=(N_{\xi},N_{\eta}) = (N_1,N_2)$ are the polynomial orders in the 2 directions of the reference element. However, all the statements and proofs can be directly generalized to $d$ dimensions.\\

\subsubsection*{\textbf{Assumptions}}
Following assumptions are a consequence of the tensor product basis functions of the DGSEM and hold for sufficiently smooth solutions in the asymptotic range. The assumptions are:

\begin{enumerate}[label=(\alph*)]
\item The truncation error has an anisotropic behavior and, therefore, can be  decoupled in its directional components:
	\begin{equation} \label{eq:TruncErrorAnisBehav}
	\tau^{N_1N_2} \approx \tau_1^{N_1N_2} + \tau_2^{N_1N_2}.
	\end{equation}
	Here, it is important to note that $\tau_i$ is the projection of the global truncation error, $\tau$, into the local direction, $i$.
\item The \textit{locally-generated} truncation error in each direction depends only on the polynomial order in that direction:
	\begin{equation} \label{eq:TruncErrorAnisBehavior}
	\tau_{\Omega,i}^{N_1N_2} \approx \tau_{\Omega,i}^{N_1N_2}(N_i)
	\end{equation}
\end{enumerate}

Assumptions (a) and (b) follow from the work of Rubio et al. \cite{Rubio2015,rubio2015truncation}. Furthermore, assumption (b) relates to the anisotropic spectral convergence behavior of the truncation error (equations \ref{eq:TruncErrorAnisConv} and  \ref{eq:IsolTruncErrorAnisConv}).

\begin{Theorem} \label{theo:NewTauEstimator}
The truncation error of a DGSEM discretization of order $(N_1,N_2)$ can be approximated from a \textit{semi}-converged solution of order $(P_1,P_2)$, such that $P_i > N_i$, as the sum of the directional $\tau$-estimations obtained by coarsening in the different space dimensions:

\begin{equation}\label{eq:NewTauEstimator}
\tau^{N_1N_2} \approx \tau^{N_1P_2}_{P_1P_2} +  \tau^{P_1N_2}_{P_1P_2}
\end{equation}
\end{Theorem}

\begin{proof}

This proof is specific for the \textit{isolated} truncation error. We refer to Appendix \ref{sec:ProofNonIsolated} for a brief proof that Theorem \ref{theo:NewTauEstimator} also holds for the \textit{non-isolated} truncation error under additional assumptions. \\

Let us note that assumptions (a) and (b) are consistent with the dependence of the \textit{isolated} truncation error on the interpolation error (equation \ref{eq:IsolTruncErrorInterpol}) and the anisotropic behavior of the latter (equation \ref{eq:AnisInterpErrorConvergence}).\\

We start by obtaining the analytical expression for the \textit{isolated} $\tau$-estimation. To that end, we use the same procedure as in Appendix \ref{sec:IsolTruncError}. The estimate of the \textit{isolated} truncation error in the DGSEM can be expressed for any basis function $\phi$ in an element $e$ as

\begin{equation}\label{eq:IsolTauEstim0}
\hat \tau^N_P \bigr\rvert_{\Omega^e} = \hat{\mathcal{R}} (\mathbf{I}^N \bar{\mathbf{q}}^P) = \int^N_{\Omega^e} \mathbf{s}^N \phi d\Omega + \int^N_{\Omega^e} \mathscr{F}^N (\bar{\mathbf{q}}^P) \cdot \nabla \phi d \Omega - \int^N_{\partial \Omega^e} \mathscr{F}^N (\bar{\mathbf{q}}^P) \cdot \mathbf{n}  \phi d \sigma,
\end{equation}

In this case, instead of the exact solution to the problem, $\bar{\mathbf{q}}$, we use a solution on a higher order mesh, $\bar{\mathbf{q}}^P$. Therefore, using the definition of interpolation error (equation \ref{eq:InterpError}) and discretization error (equation \ref{eq:DiscError}), the flux is

\begin{equation} \label{eq:IsolTauEstim1}
\mathscr{F}^N = \mathbf{I}^N \mathscr{F}(\bar{\mathbf{q}}^P) = \mathscr{F}(\bar{\mathbf{q}}^P) - \varepsilon^N_{\mathscr{F}} = \mathscr{F}(\bar{\mathbf{q}}) + \frac{\partial \mathscr{F}}{\partial \mathbf{q}} \biggr\rvert_{\bar{\mathbf{q}}} {\epsilon^P} - \varepsilon^N_{\mathscr{F}} + \mathcal{O} \left( {(\epsilon^P)^2} \right),
\end{equation}
and the source term is
\begin{equation}\label{eq:IsolTauEstim2}
\mathbf{s}^N = \mathbf{I}^N \mathbf{s} = \mathbf{s} - \varepsilon^N_{\mathbf{s}}.
\end{equation}

Inserting equations \ref{eq:IsolTauEstim1} and \ref{eq:IsolTauEstim2} into \ref{eq:IsolTauEstim0}, and integrating by parts we obtain

\begin{equation} \label{eq:IsolTauEstim}
\hat \tau_{P}^{N} \biggr\rvert_{\Omega^e} = \left( \nabla \cdot \varepsilon^{N}_{\mathscr{F}} , \phi \right)^{N}_{\Omega^e} + \left( \nabla \cdot \frac{\partial \mathscr{F}}{\partial \mathbf{q}} \biggr\rvert_{\bar{\mathbf{q}}} \epsilon^{P} , \phi \right)^{N}_{\Omega^e}  + \mathcal{O} \left( (\epsilon^{P})^2 \right) + \mathcal{O} \left( e_{\int}^{N} \right).
\end{equation}

Remark that although the \textit{isolated} truncation error of an element does not depend on external sources, its approximation by $\tau$-estimation is affected by the discretization error of the reference mesh, $\epsilon^{P}$. This translates into a weak influence of external (upwind) errors transmitted through the Riemann solver. In the two-dimensional case and coarsening in only one direction (here the direction $(1)$), equation \ref{eq:IsolTauEstim} becomes

\begin{multline} \label{eq:DirectionalIsolTau1}
\hat \tau_{P_1P_2}^{N_1P_2} \biggr\rvert_{\Omega^e} = \left( \nabla \cdot \varepsilon^{N_1P_2}_{\mathscr{F}} , \phi \right)^{N_1P_2}_{\Omega^e} + \left( \nabla \cdot \frac{\partial \mathscr{F}}{\partial \mathbf{q}} \biggr\rvert_{\bar{\mathbf{q}}} \epsilon^{P_1P_2} , \phi \right)^{N_1P_2}_{\Omega^e}  \\
+ \mathcal{O} \left( (\epsilon^{P_1P_2})^2 \right) + \mathcal{O} \left( e_{\int}^{N_1P_2} \right).
\end{multline}

Now, we rewrite equation \ref{eq:DirectionalIsolTau1} decoupling the interpolation error in directional components and taking into account that $\varepsilon_2^{N_1P_2} = \varepsilon_2^{P_1P_2}$ (equation \ref{eq:DirectionalInterpolError}):

\begin{multline} \label{eq:DirectionalIsolTau2}
\hat \tau_{P_1P_2}^{N_1P_2} = \left( \nabla \cdot \varepsilon^{N_1P_2}_{\mathscr{F},1} , \phi \right)^{N_1P_2}_{\Omega^e} + \left( \nabla \cdot \varepsilon^{P_1P_2}_{\mathscr{F},2} , \phi \right)^{N_1P_2}_{\Omega^e} + \\
\left( \nabla \cdot \frac{\partial \mathscr{F}}{\partial \mathbf{q}} \biggr\rvert_{\bar{\mathbf{q}}} \epsilon^{P_1P_2} , \phi \right)^{N_1P_2}_{\Omega^e}  + \mathcal{O} \left( (\epsilon^{P_1P_2})^2 \right) + \mathcal{O} \left( e_{\int}^{N_1P_2} \right).
\end{multline}

Notice that all terms on the right-hand side, except for the first one and the quadrature error, are of the order of errors on the higher-order mesh. Therefore, and taking into account that we are coarsening in the direction $(1)$, for sufficiently smooth solutions we can expect the first term on the right-hand side to be the leading term. Simplifying, the directional $\hat \tau$-estimation provides

\begin{equation} \label{eq:IsolTauDirectional}
\hat \tau_{P_1P_2}^{N_1P_2} \approx \left( \nabla \cdot \varepsilon^{N_1P_2}_{\mathscr{F},1} , \phi \right)^{N_1P_2}_{\Omega^e}.
\end{equation}

On the other hand, inserting equation \ref{eq:DirectionalInterpolError} into \ref{eq:IsolTruncErrorInterpol}, the \textit{isolated} truncation error of a representation of order $N=(N_1,N_2)$ yields

\begin{equation} \label{eq:IsolTauDecoupled}
\hat \tau^{N_1N_2} \bigr\rvert_{\Omega^e} \approx \sum_{i=1}^d \left( \nabla \cdot \varepsilon^{N_1N_2}_{\mathscr{F},i} , \phi \right)^{N_iN_j}_{\Omega^e}.
\end{equation} 

Notice, again, that the directional components of the interpolation error only depend on the polynomial order in the corresponding direction (equation \ref{eq:DirectionalInterpolError}). Therefore, neglecting additional quadrature errors, we recover equation \ref{eq:NewTauEstimator} for the \textit{isolated} truncation error by combining equations \ref{eq:IsolTauDirectional} and \ref{eq:IsolTauDecoupled}:

\begin{equation}
\hat \tau^{N_1N_2} \approx \hat \tau^{N_1P_2}_{P_1P_2} + \hat \tau^{P_1N_2}_{P_1P_2}.
\end{equation}
\end{proof}

From the previous analysis we can conclude that, when the $\tau$-estimation method is performed coarsening only in the direction $i$, the result is an approximation to the $i$-directional component of the truncation error, $\tau_i$. We can also arrive at this conclusion intuitively if we realize that $\mathbf{Q}^{P_1P_2}$ cannot be better than $\mathbf{Q}^{N_1P_2}$ at describing the solution in the direction $(2)$.\\

Theorem \ref{theo:NewTauEstimator} can be easily generalized to three dimensions to obtain

\begin{align*}
\tau^{N_1N_2N_3} &\approx \tau_1^{N_1N_2N_3} + \tau_2^{N_1N_2N_3} + \tau_3^{N_1N_2N_3} \\
\tau^{N_1N_2N_3} &\approx \tau_{P_1P_2P_3}^{N_1P_2P_3} + \tau_{P_1P_2P_3}^{P_1N_2P_3} + \tau_{P_1P_2P_3}^{P_1P_2N_3}.
\end{align*}

\subsection{Convergence behavior of the anisotropic truncation error} \label{sec:AnisConv}

In this section we analyze the convergence properties of the truncation error map using Theorem \ref{theo:NewTauEstimator}. Let us first consider the directional components of the truncation error.

%%%%%%%%%%%%%%%%%%%%%%%%%%%%%%%%%%%%%%%%%%%%%%%%%%%%%%%%%%%%%%%%%%%%%%%
%%%%%%%%%%% NEW THEOREM %%%%%%%%%%%%%%%%%%%%%%%%%%%%%%%%%%%%%%%%%%%%%%

\begin{Theorem} \label{theo:AnisTruncErrorConvergence}
The directional components of the \textit{locally-generated} truncation error exhibit spectral convergence with respect to the polynomial order in the corresponding direction:
\begin{equation} \label{eq:AnisTruncErrorConvergence}
\norm{\tau_{\Omega^e,i}} \le  C^e_{i} e^{-N^e_{i} \eta^e_{i}} 
\end{equation}
\end{Theorem}

\begin{proof}
According to assumption $(b)$ (equation \ref{eq:TruncErrorAnisBehavior}), each directional component of the \textit{locally-generated} truncation error, $\tau_{\Omega,i}$, depends solely on the polynomial order in the corresponding direction, $N_i$. If we insert equation \ref{eq:TruncErrorAnisBehavior} into equation \ref{eq:IsolTruncErrorAnisConv} and analyze the dependencies term by term, we recover \ref{eq:AnisTruncErrorConvergence} for the \textit{isolated} truncation error. In the same way, if we insert equation \ref{eq:TruncErrorAnisBehavior} into \ref{eq:TruncErrorAnisConv}, we recover \ref{eq:AnisTruncErrorConvergence} for the \textit{non-isolated} truncation error.\\
\end{proof}

%%%%%%%%%%%%%%%%%%%%%%%%%%%%%%%%%%%%%%%%%%%%%%%%%%%%%%%%%%%%%%%%%%%%%%%

Now, we are able to analyze the convergence behavior along lines of the truncation error map, where a line in $d$ dimensions is defined as:

\begin{equation}
a_1N_1+a_2N_2+\cdots+a_dN_d=b, 
\end{equation}
with $a_i,b \in {\rm I\!R}$.

%%%%%%%%%%%%%%%%%%%%%%%%%%%%%%%%%%%%%%%%%%%%%%%%%%%%%%%%%%%%%%%%%%%%%%%
%%%%%%%%%%% NEW THEOREM %%%%%%%%%%%%%%%%%%%%%%%%%%%%%%%%%%%%%%%%%%%%%%
\begin{Theorem} \label{theo:NotSpectralOnLines}
The total truncation error does not necessarily decrease exponentially along lines of the truncation error map.

\end{Theorem}

\begin{proof}
The corresponding positive statement can be easily proven wrong with a counterexample. Let us consider a 2D anisotropic representation. From theorem \ref{theo:AnisTruncErrorConvergence}, we know that the truncation error of each directional component in an element, $\tau_i^N$, decreases exponentially when increasing $N_i$; and that the decreasing rate is $\eta_i$, a constant that depends on the smoothness of the solution in the direction $i$. Let us suppose that for a certain element in a mesh, the directional components of the error have the same value for a specific combination of polynomial orders $(\bar N_1, \bar N_2)$ in a certain norm:

\begin{equation} \label{eq:lines1}
\norm{\tau_1^{\bar N_1\bar N_2}} = \norm{\tau_2^{\bar N_1\bar N_2}} = C.
\end{equation}

Let us analyze the convergence rate along an iso-$N_i$ line of the truncation error map with constant $N_1 = \bar N_1$, i.e. the convergence rate of $\tau^{\bar N_1N_2}$ as a function of $N_2$. Note that, according to assumption (b), along the line we have
\begin{align} \label{eq:lines1.5}
\tau_1^{\bar N_1N_2} &\approx \tau_1^{\bar N_1 \bar N_2} \nonumber\\
\norm{\tau_1^{\bar N_1N_2}} &\approx \norm{\tau_1^{\bar N_1 \bar N_2}} = C.
\end{align}

Furthermore, assumption (a) states that the total truncation error along the line of the map is:

\begin{equation} \label{eq:lines1.7}
\tau^{\bar N_1N_2} \approx \tau_1^{\bar N_1N_2} + \tau_2^{\bar N_1N_2}.
\end{equation}

Substituting equation \ref{eq:lines1.5} into \ref{eq:lines1.7} and using the triangle inequality yields:

\begin{equation} \label{eq:lines2}
\norm{\tau^{\bar N_1N_2}} \le C + \norm{\tau_2^{\bar N_1N_2}}.
\end{equation}

Remember that the truncation error map is defined as the dependence of $\log \norm{\tau^{N_1N_2}}$ on $(N_1,N_2)$ (definition \ref{Def:TruncErrorMap}). Therefore, taking logarithms in both sides and rearranging, equation \ref{eq:lines2} can be rewritten in two equivalent forms:
\begin{eqnarray}
i) & &
\log \norm{\tau^{\bar N_1N_2}} \le \log \norm{\tau_2^{\bar N_1N_2}} + \log \left( 1 + \frac{ C }  {\norm{\tau_2^{\bar N_1N_2}} } \right), \label{eq:linesLowN2} \\
ii) & &
\log \norm{\tau^{\bar N_1N_2}} \le \log (C) + \log \left( 1 + \frac{ \norm{\tau_2^{\bar N_1N_2}} }  { C } \right) \label{eq:linesHighN2}.
\end{eqnarray}

For $N_2 \ll \bar N_2$, the second term on the right-hand side of equation \ref{eq:linesLowN2} vanishes, which indicates that the convergence rate of the truncation error along an iso-$N_i$ line of constant $N_1 = \bar N_1$ tends to $\eta_2$. On the other hand, for $N_2 \gg \bar N_2$, the second term on the right-hand side of equation \ref{eq:linesHighN2} vanishes, which implies that the convergence rate along an iso-$N_i$ line of constant $N_1 = \bar N_1$ tends to zero, since the  truncation error is bounded by $\log (C)$, a constant that does not depend on $N_2$. In other words, the iso-$N_i$ line on the hyperplane for $N_1=\bar N_1$ is not a straight line.\\
\end{proof}
%%%%%%%%%%%%%%%%%%%%%%%%%%%%%%%%%%%%%%%%%%%%%%%%%%%%%%%%%%%%%%%%%%%%%%%

\subsection{High order extrapolation of the truncation error estimations}\label{sec:NewPolOrders}

In this section, we present a procedure for extrapolating the truncation error estimations (\textit{inner} map) that can be obtained by applying Theorem \ref{theo:NewTauEstimator}. Since Theorem \ref{theo:NotSpectralOnLines} rules out the possibility of extrapolating along iso-$N_i$ lines of the truncation error map, we take advantage of the anisotropic behavior of the truncation error (equation \ref{eq:TruncErrorAnisBehav}) and the spectral convergence of its directional components (Theorem \ref{theo:AnisTruncErrorConvergence}). The proposed methodology, which is valid for both the \textit{isolated} and \textit{non-isolated} truncation error, can be summarized in three steps:

\begin{enumerate}
\item Perform anisotropic coarsening to obtain $\tau_i$ (Theorem \ref{theo:NewTauEstimator}) and construct the \textit{inner} truncation error map directly. In $d$ dimensions, this requires only $n^{new}_{eval}$ evaluations of the discrete partial differential operator $\mathcal{R}^N$, where
\begin{equation}
n^{new}_{eval} = \sum_{i=1}^d (P_i-1).
\end{equation}

\item Compute $\log \norm{\tau_i}$ and perform a linear regression analysis in the direction $i$ in order to describe the behavior of $\log \norm{\tau_i}$ as a function of $N_i$. This is supported on the proved spectral behavior of the directional components of the truncation error (Theorem \ref{theo:AnisTruncErrorConvergence}).

\item Construct the \textit{outer} truncation error map using equation \ref{eq:NewTauEstimator} and the extrapolated values of $\log \norm{\tau_i}$:
\begin{equation}
\log \norm{\tau} = \log \norm{\sum_{i=1}^d \tau_i}.
\end{equation}
\end{enumerate}

An example is provided in section \ref{sec:Validation}.

\subsection{Theoretical comparison of the new anisotropic $\tau$-estimation with previous approaches} \label{sec:TheoComparisonMaps}

Figure \ref{fig:AnisTauBehavior} illustrates the theoretically predicted behavior of the truncation error map that is obtained with the new anisotropic $\tau$-estimation method for a toy problem (only illustrative). As noted in the proof of Theorem \ref{theo:NotSpectralOnLines}, remark that along an iso-$N_i$ line, the truncation error firstly decays exponentially for low $N_{j \ne i}$ and then tends asymptotically to a constant value for high $N_{j \ne i}$. As can be seen, contrary to the \textit{low-order extrapolation}, the new anisotropic $\tau$-estimation does not assume that the truncation error map has a linear behavior. That is the reason why it is called \textit{high-order extrapolation}. Figure \ref{fig:Overlap} shows a comparison of the hyperplane behavior with the one obtained using the new anisotropic $\tau$-estimation method. As can be seen, the hyperplane tends to underpredict the truncation error for some combinations of polynomial orders as compared to the new truncation error estimator. A comparison of the output of both estimation methods with the exact truncation error in a test case is provided in section \ref{sec:ComparisonKompen}.\\

\begin{figure}[h!]
\begin{center}

\subfigure[New anisotropic $\tau$-estimation method]{\label{fig:AnisTauBehavior} \includegraphics[width=0.45\textwidth]{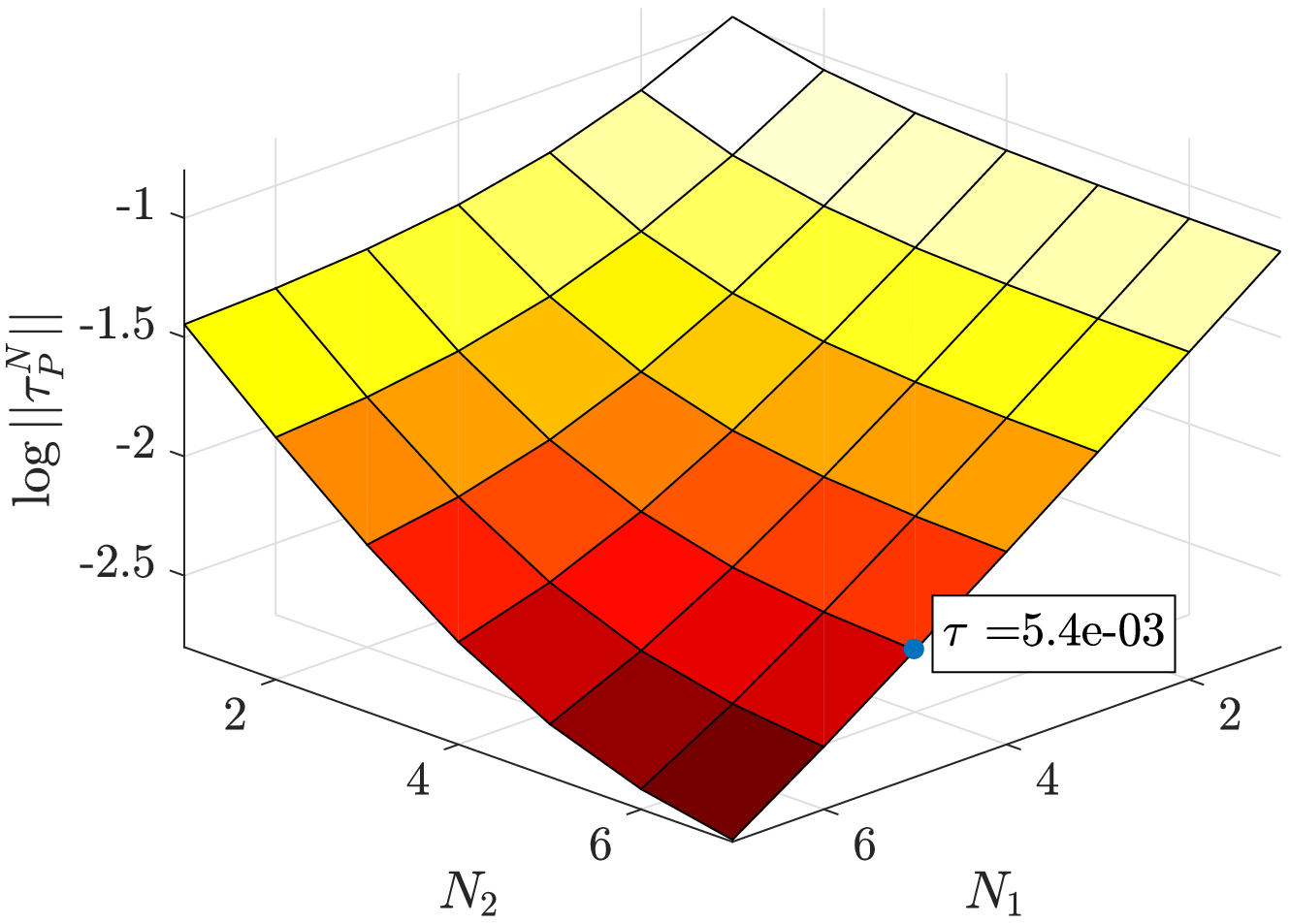}}\qquad
\subfigure[Hyperplane overlapped with new anisotropic $\tau$-estimation method]{\label{fig:Overlap} \includegraphics[width=0.45\textwidth]{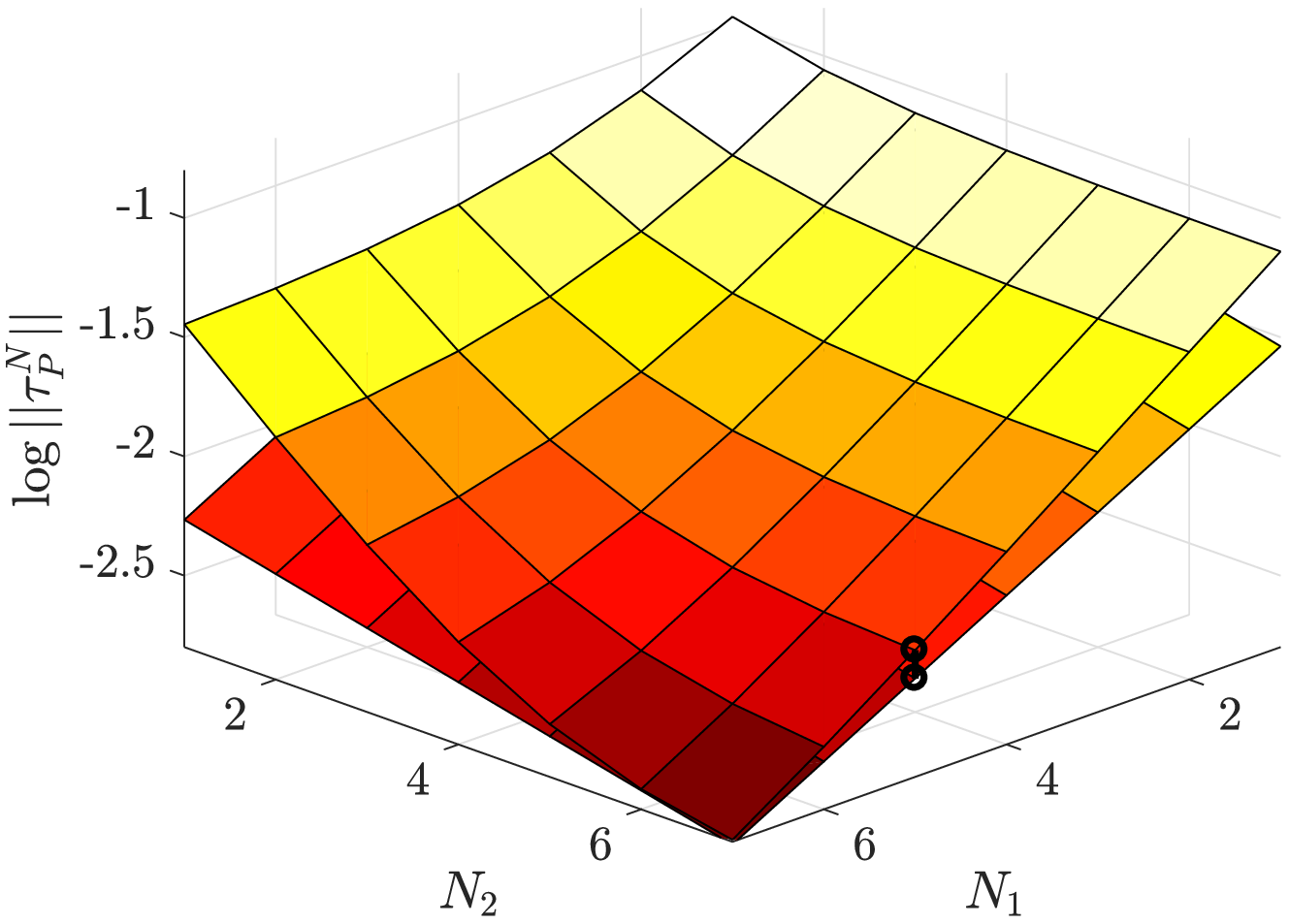}}
\caption{Spatial representation of two-dimensional anisotropic truncation error maps: behavior predicted by the new anisotropic $\tau$-estimation  method (a), hyperplane behavior (Figure \ref{fig:HyperPlane}) overlapped with the values predicted by the new anisotropic $\tau$-estimation method (b)} \label{fig:Surfaces}

\end{center}
\end{figure}

It is noteworthy that even though the method of Kompenhans et al. \cite{Kompenhans2016} supposes hyperplane behavior, their strategy of selecting the polynomial order in every direction independently minimizes the error involved in the \textit{low order extrapolation} (compare the values of $\tau$ in Figures \ref{fig:AnisTauBehavior} and \ref{fig:HyperPlane}). Furthermore, remark that for $N_1 \gg N_2$ and $N_2 \gg N_1$ the new anisotropic estimation tends to have a hyperplane behavior. In fact, Kompenhans et al. \cite[section~5.1]{Kompenhans2016} state that only the values of the truncation error where $N_1 \gg N_2$ or $N_2 \gg N_1$ should be used for the least square fitting. The high-order extrapolation can be seen as a form of bypassing this requirement.\\

Table \ref{tab:Comparison} provides a final summary comparing the new estimation method and the previous methodology by Kompenhans et al. \cite{Kompenhans2016}.

\begin{table}[h]
\caption{Comparison of anisotropic $\tau$-estimation methods in $d$ dimensions performed with a reference mesh of order $(P_1, \cdots, P_d)$. $\mathcal{R}^N$ is the discrete partial differential operator.}
\label{tab:Comparison} 
\begin{tabular}{p{3cm}cc}
\hline\noalign{\smallskip}
Feature & Kompenhans et al. \cite{Kompenhans2016}  & Proposed $\tau$-estimation\\
\noalign{\smallskip}\hline\noalign{\smallskip}
Number of evaluations of $\mathcal{R}^N$ for inner map & 
\(\displaystyle
\prod_{i=1}^d (P_i-1)
\)
 &   
$\displaystyle \sum_{i=1}^d (P_i-1)$ \\

Accuracy of inner map & Very good & Good\\
Accuracy of outer map & Poor & Good \\
\noalign{\smallskip}\hline
\end{tabular}
\end{table}

\section{Validation of the anisotropic $\tau$-estimation method} \label{sec:Validation}

The compressible Navier-Stokes equations can be written in conservative form (see Appendix \ref{sec:NS}) and discretized using the DGSEM, as explained in section \ref{sec:DGSEM}. In order to test the accuracy of the proposed $\tau$-estimation method, a 2D manufactured solutions test case is analyzed. The exact solution selected for the problem is
\begin{align} \label{eq:ManufacturedSolution}
\rho &= p = e^{-5 \left( 4(x-\frac{1}{2})^2 + (y-\frac{1}{2})^2  \right)} + 1, \nonumber \\
u    &= v = 1,
\end{align}
which is simulated in the unit square, as depicted in Figure \ref{fig:ValidationRho}. Inserting equation \ref{eq:ManufacturedSolution} into \ref{eq:NSeq}, the source term for the 2D compressible Navier-Stokes equations yields

\begin{equation}
\mathbf{s} = 
\begin{bmatrix}
s_{\rho} \\ s_{\rho u} \\ s_{\rho v} \\ s_{\rho e}
\end{bmatrix} =
\begin{bmatrix}

40 \left(x-\frac{1}{2}\right) + 10 \left(y-\frac{1}{2}\right) \\

80 \left(x-\frac{1}{2}\right) + 10 \left(y-\frac{1}{2}\right) \\

40 \left(x-\frac{1}{2}\right) + 20 \left(y-\frac{1}{2}\right) \\

\left[40 \left(x-\frac{1}{2}\right) + 10 \left(y-\frac{1}{2}\right) \right]
\left[\frac{1}{\gamma -1}+2\right] 

\end{bmatrix} e^{5 \left(4 (x-\frac{1}{2})^2+(y-\frac{1}{2})^2\right)}.
\end{equation}

\begin{figure}[h]
\begin{center}
\includegraphics[width=0.4\textwidth]{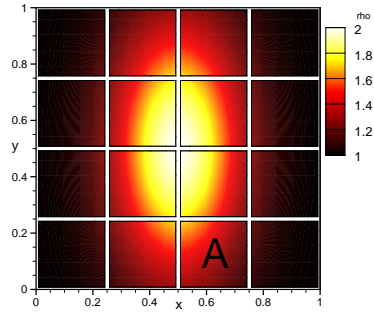}
\caption{Density ($\rho$) contours for the proposed manufactured solutions test case}\label{fig:ValidationRho}
\end{center}
\end{figure}

The main interest is to validate the proposed error estimator and compare its outcome with previous works. Since the method of Kompenhans et al. \cite{Kompenhans2016} explained in section \ref{sec:OldPolOrders} was formulated and used with the \textit{non-isolated} truncation error, the results that are shown in sections \ref{sec:TEmaps} and \ref{sec:ComparisonKompen} were obtained for the \textit{non-isolated} truncation error estimator ($\tau_P^N$). However, similar results can be obtained for the \textit{isolated} truncation error, since its maps exhibit the same behavior as the ones presented here. In addition, in section \ref{sec:ResIsolTruncError} we compare the truncation error estimator with the \textit{isolated} truncation error estimator when used for driving a p-adaptation procedure.

\subsection{Truncation error maps and number of degrees of freedom}\label{sec:TEmaps}
A fully time-converged solution ($||\mathcal{R}^P||_\infty<10^{-10}$) of order 5 ($P=P_1=P_2=5$) is used to estimate the truncation error using the method of section \ref{sec:AnisTauEstimation}. The results for element \textit{A} are depicted in Figure \ref{fig:ValidationTauEstim}. Figure \ref{fig:ValidationTauEx} shows the exact truncation error. It can be seen that the proposed method predicts a truncation error map that is very similar to the exact one, even for extrapolated values. Hence, in agreement with the obtained results, the assumptions of section \ref{sec:AnisTauEstimation} are reasonable. These maps can be used for selecting an appropriate combination of polynomial orders such that a maximum truncation error $\tau_{max}$ is achieved employing a minimum number of degrees of freedom. \\

Figure \ref{fig:ValidationDOFs} shows the map of the number degrees of freedom (DOFs) for every ($N_1$,$N_2$)-combination. The polynomial orders that achieve a truncation error $\tau<\tau_{max}$ are marked with black squares. Let us remark that, although these results are not exactly the same for the estimated and exact truncation error maps, they are very similar. Therefore, we conclude that the proposed estimation method may be used for adaptation purposes. Notice that there are many alternatives that produce a truncation error in the desired range, but there is only one that minimizes the number of degrees of freedom and, therefore, the computational cost.\\
\begin{figure}[h!]
\begin{center}

\subfigure[Estimated $\tau^N_5$ with $P_1=P_2=5$]{\label{fig:ValidationTauEstim} \includegraphics[width=0.4\textwidth]{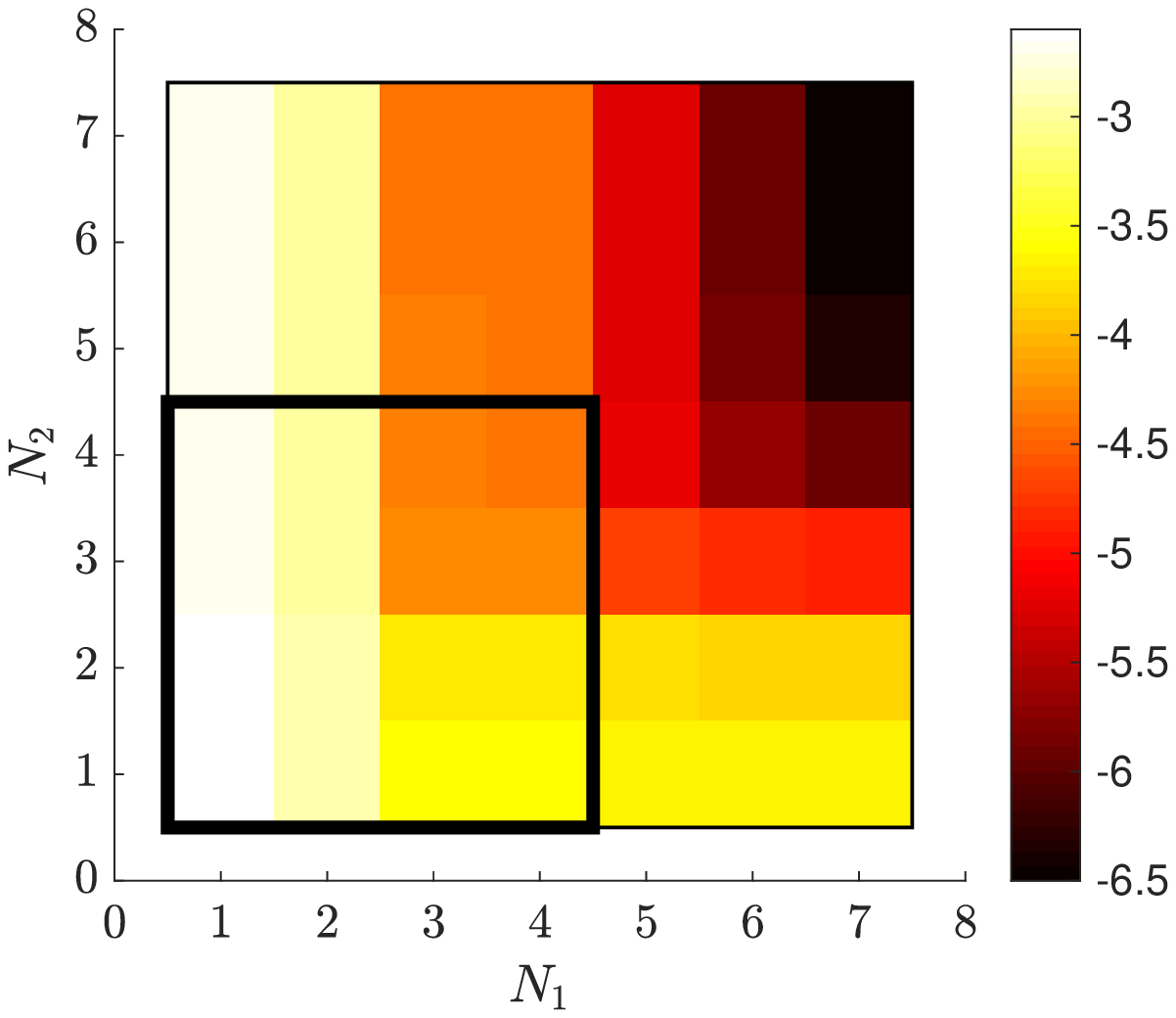}}\qquad
\subfigure[Exact truncation error]{\label{fig:ValidationTauEx} \includegraphics[width=0.4\textwidth]{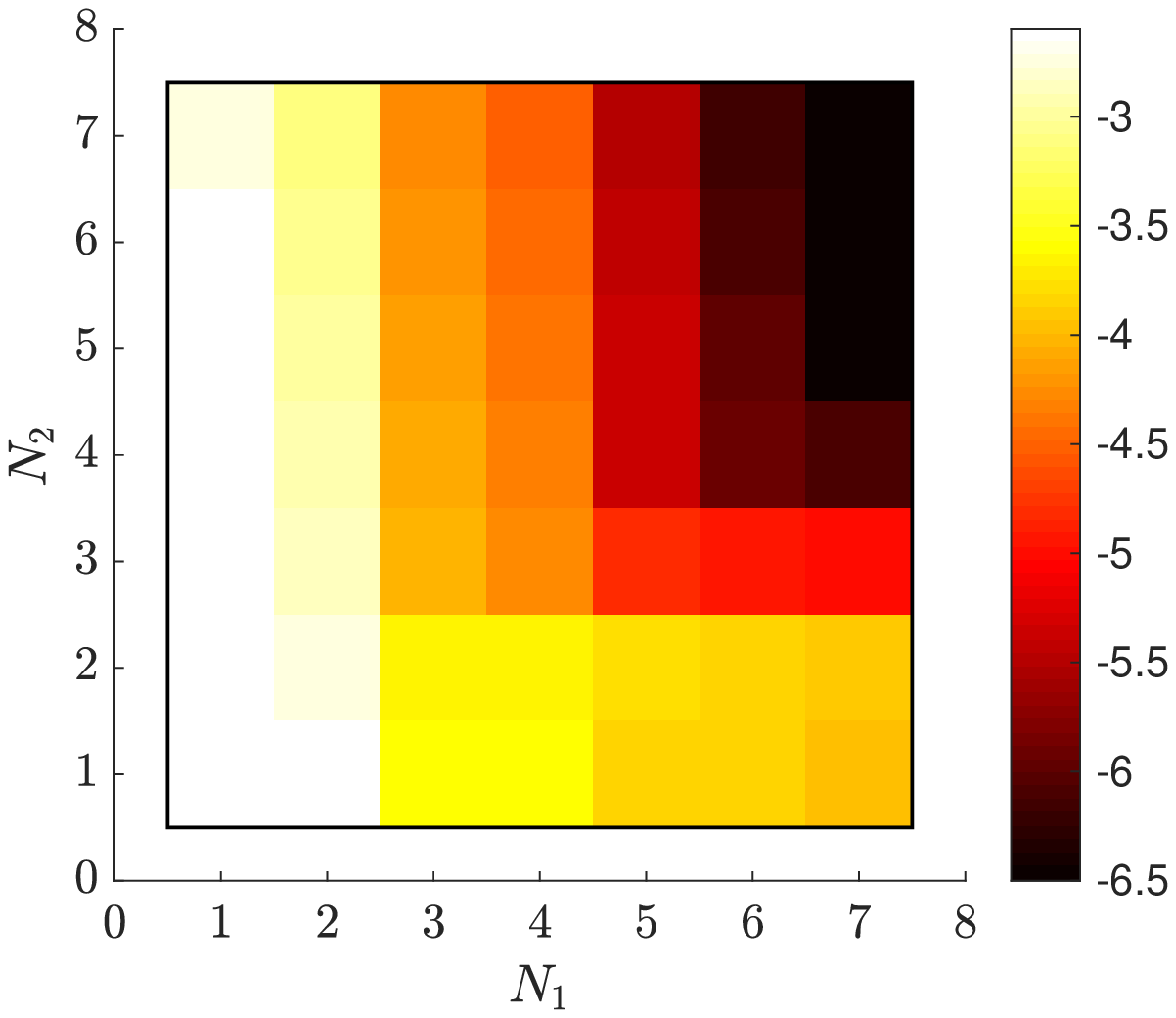}}

\caption{Truncation error estimation (a) and exact values (b) for different polynomial order combinations in element A (logarithmic scale). Outside the black box (a) are the extrapolated values of the estimated truncation error} \label{fig:ValidationTau}

\vspace*{\floatsep}% https://tex.stackexchange.com/q/26521/5764

\subfigure[$\tau_{max}=10^{-4}$ (estimation).]{\label{fig:ValidationDOFEstim10_4} \includegraphics[width=0.4\textwidth]{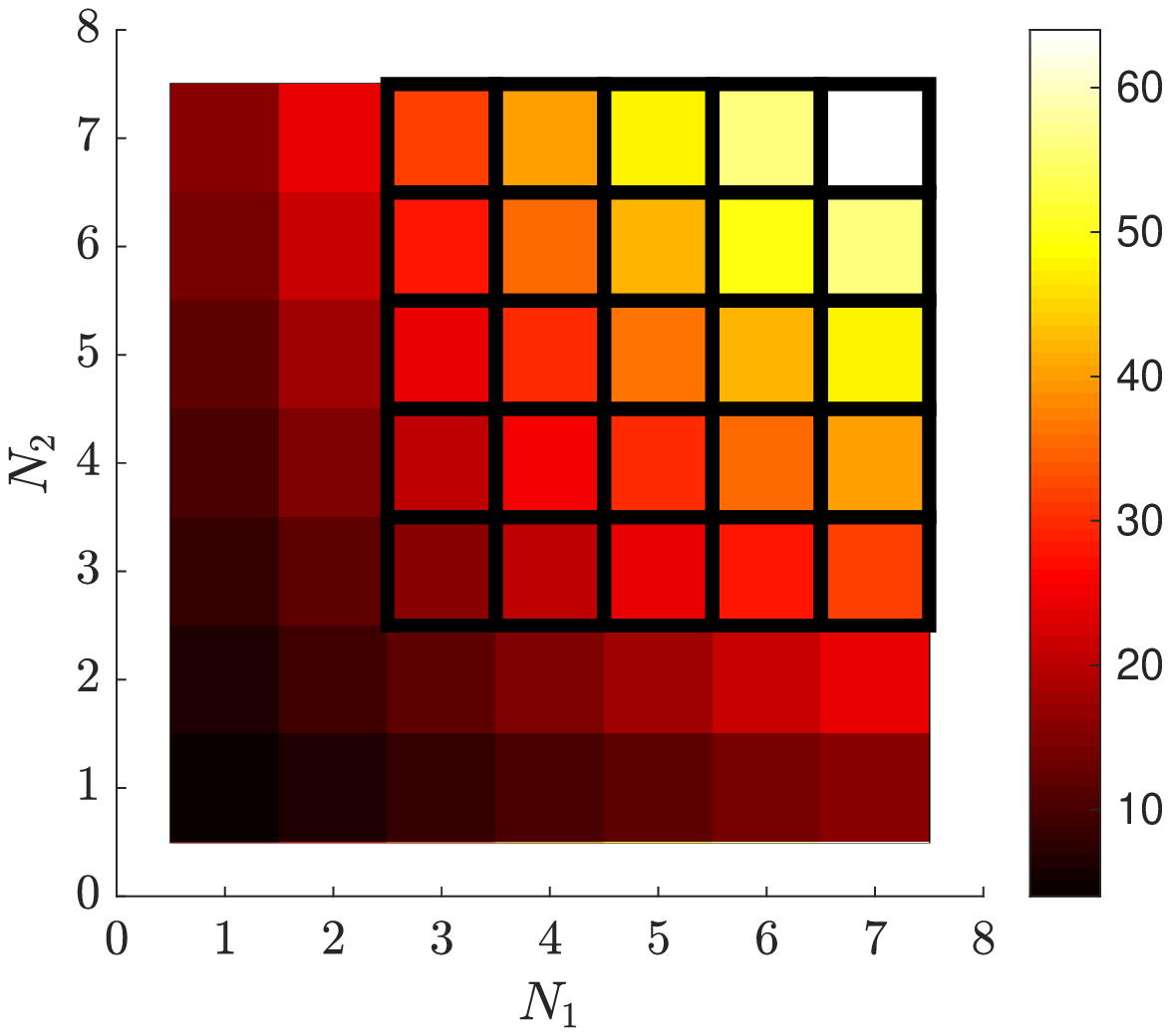}}\qquad
\subfigure[$\tau_{max}=10^{-4}$ (exact).]{\label{fig:ValidationDOFEx10_4} \includegraphics[width=0.4\textwidth]{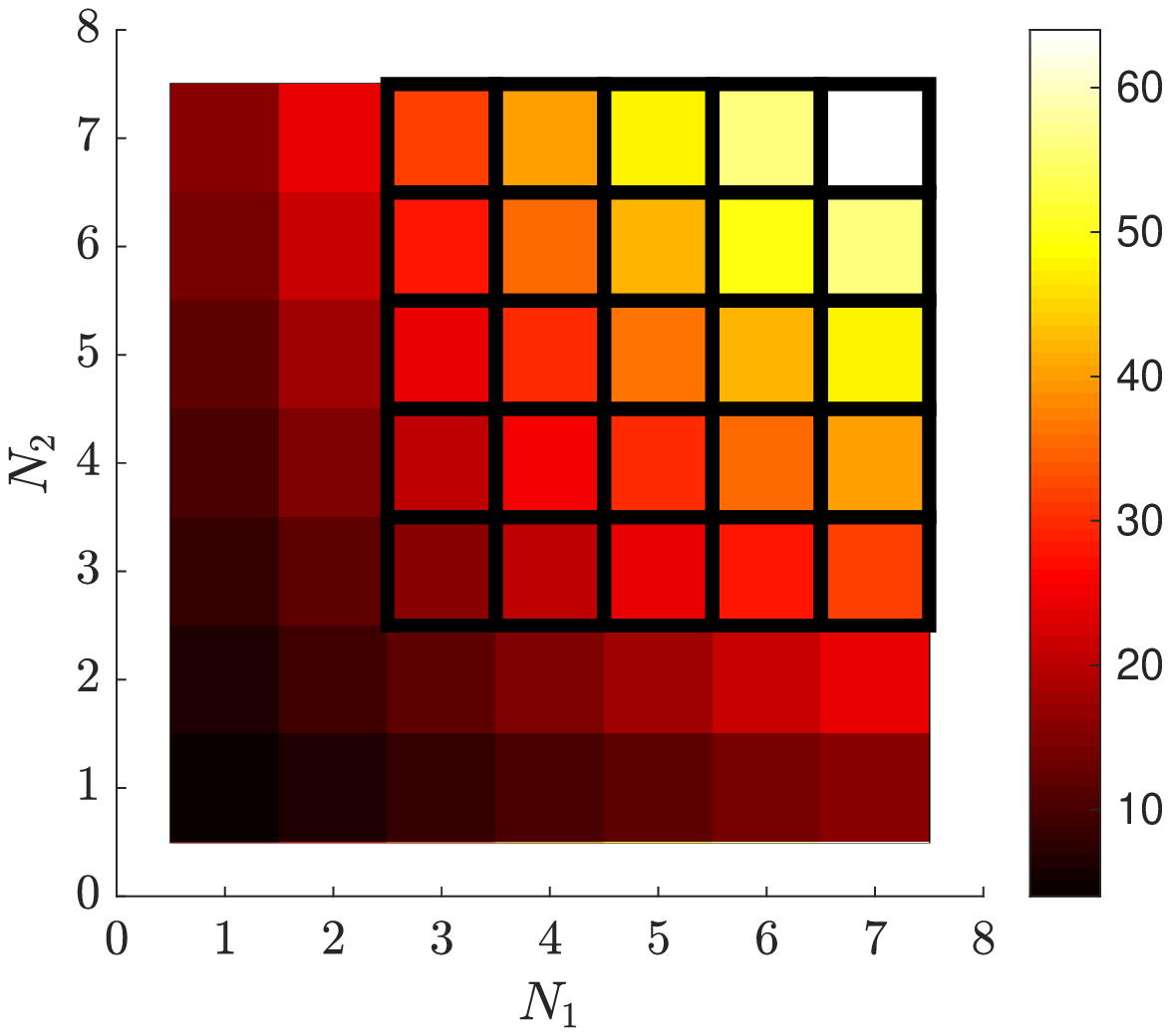}}\qquad

\subfigure[$\tau_{max}=10^{-5}$ (estimation).]{\label{fig:ValidationDOFEstim10_5} \includegraphics[width=0.4\textwidth]{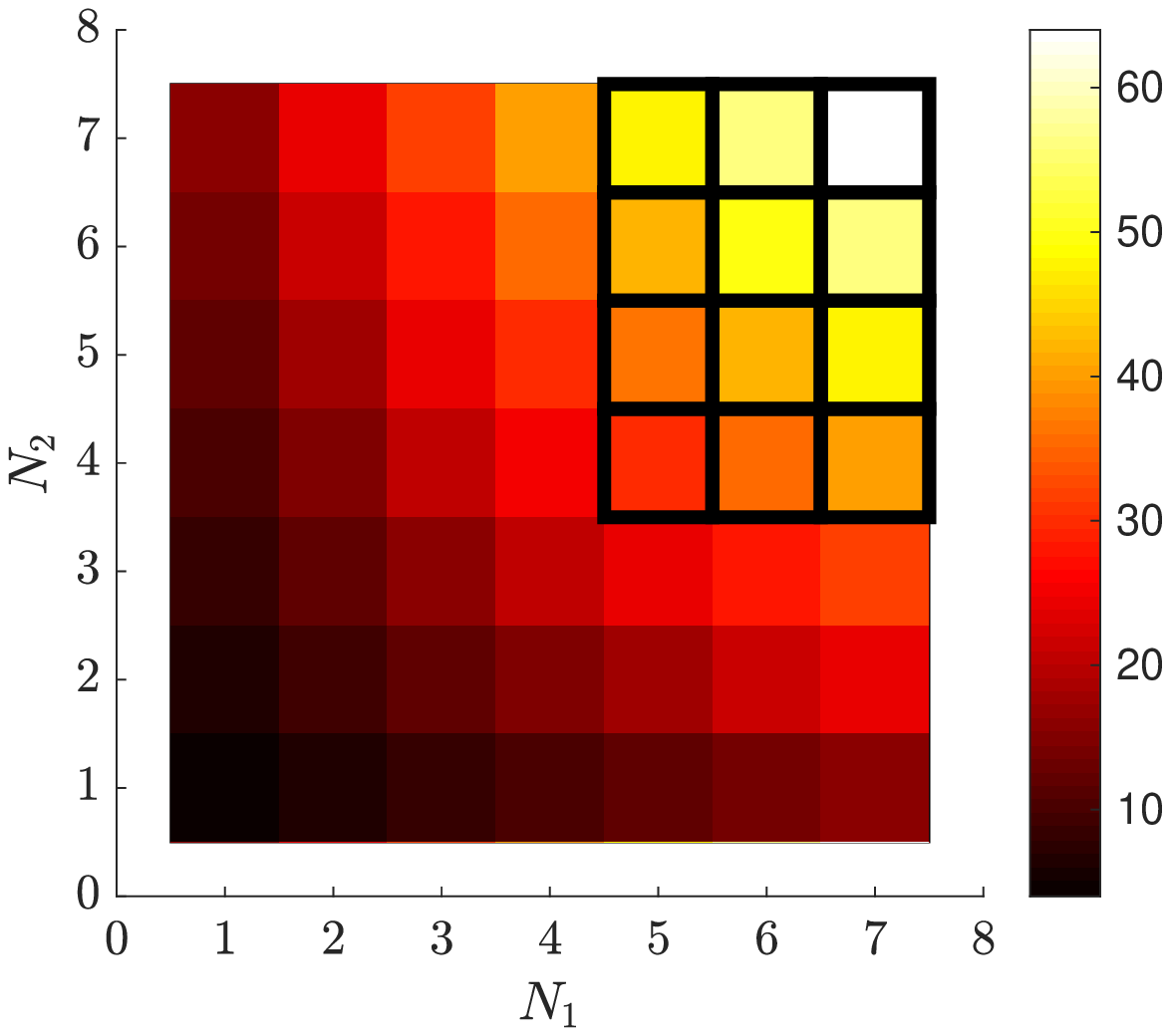}}\qquad
\subfigure[$\tau_{max}=10^{-5}$ (exact).]{\label{fig:ValidationDOFEx10_5} \includegraphics[width=0.4\textwidth]{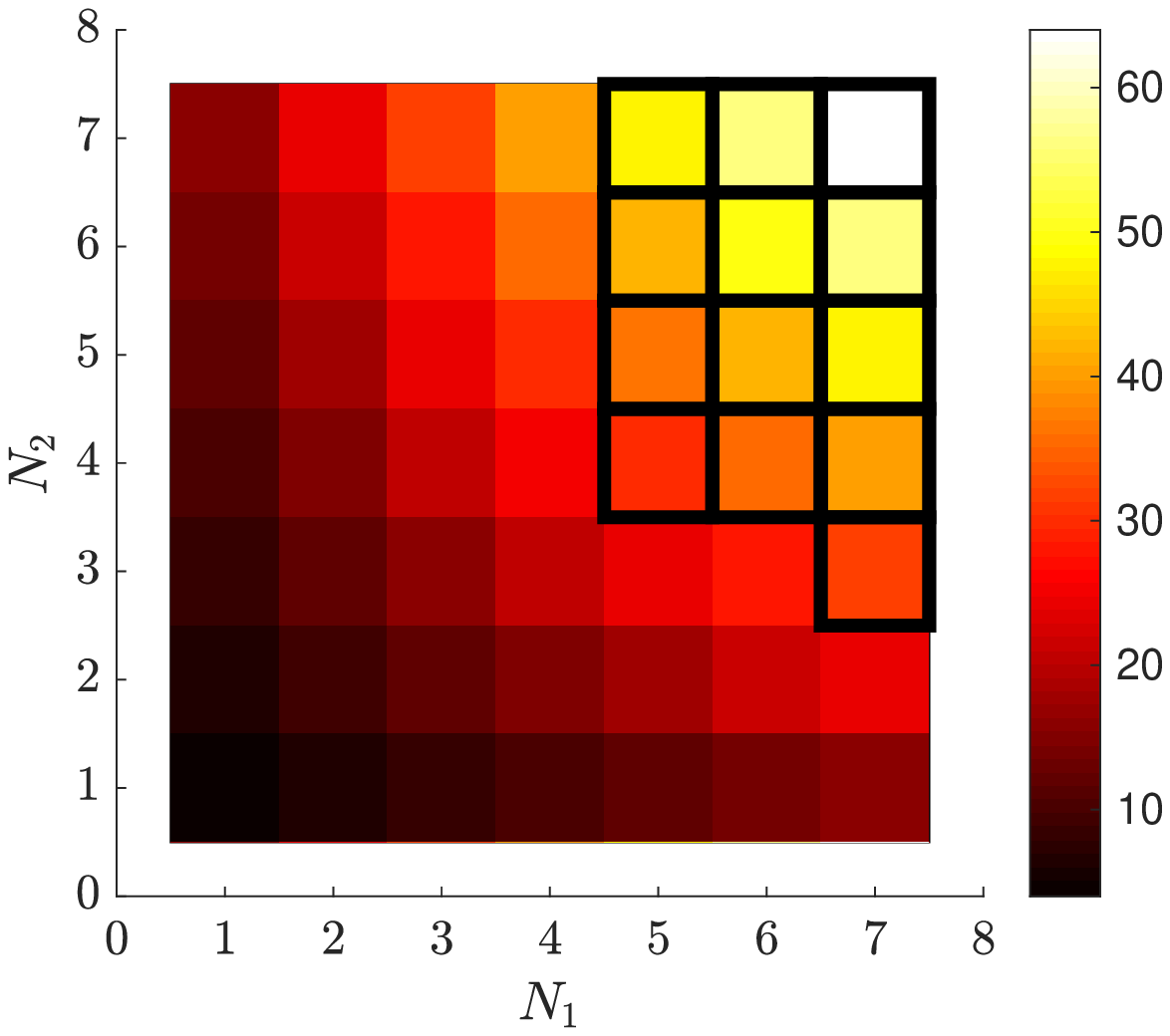}}\qquad

\caption{Contour of the number of degrees of freedom for every polynomial order considered for performing p-adaptation. The combinations $(N_1,N_2)$ that fulfill the $\tau_{max}$ threshold are marked with black squares} \label{fig:ValidationDOFs}
\end{center}
\end{figure}

\subsection{Comparison with previous methodologies}\label{sec:ComparisonKompen}

Figure \ref{fig:Res_Surfaces} shows the 3D representation of the exact truncation error map (a), the one obtained with the \textit{high-order extrapolation} (b), and the one obtained with the \textit{low-order extrapolation} (c) -here, we illustrate the complete hyperplane. The maps were generated with the same fully time-converged solution of order $P_1=P_2=5$. As can be seen, the truncation error map generated with the \textit{high-order extrapolation} bears close resemblance to the exact one, whereas the hyperplane underpredicts the truncation error in some regions, as anticipated in section \ref{sec:TheoComparisonMaps}.\\

\begin{figure}[h!]
\begin{center}

\subfigure[Exact truncation error map]{\label{fig:Res_TEmapEx} \includegraphics[width=0.45\textwidth]{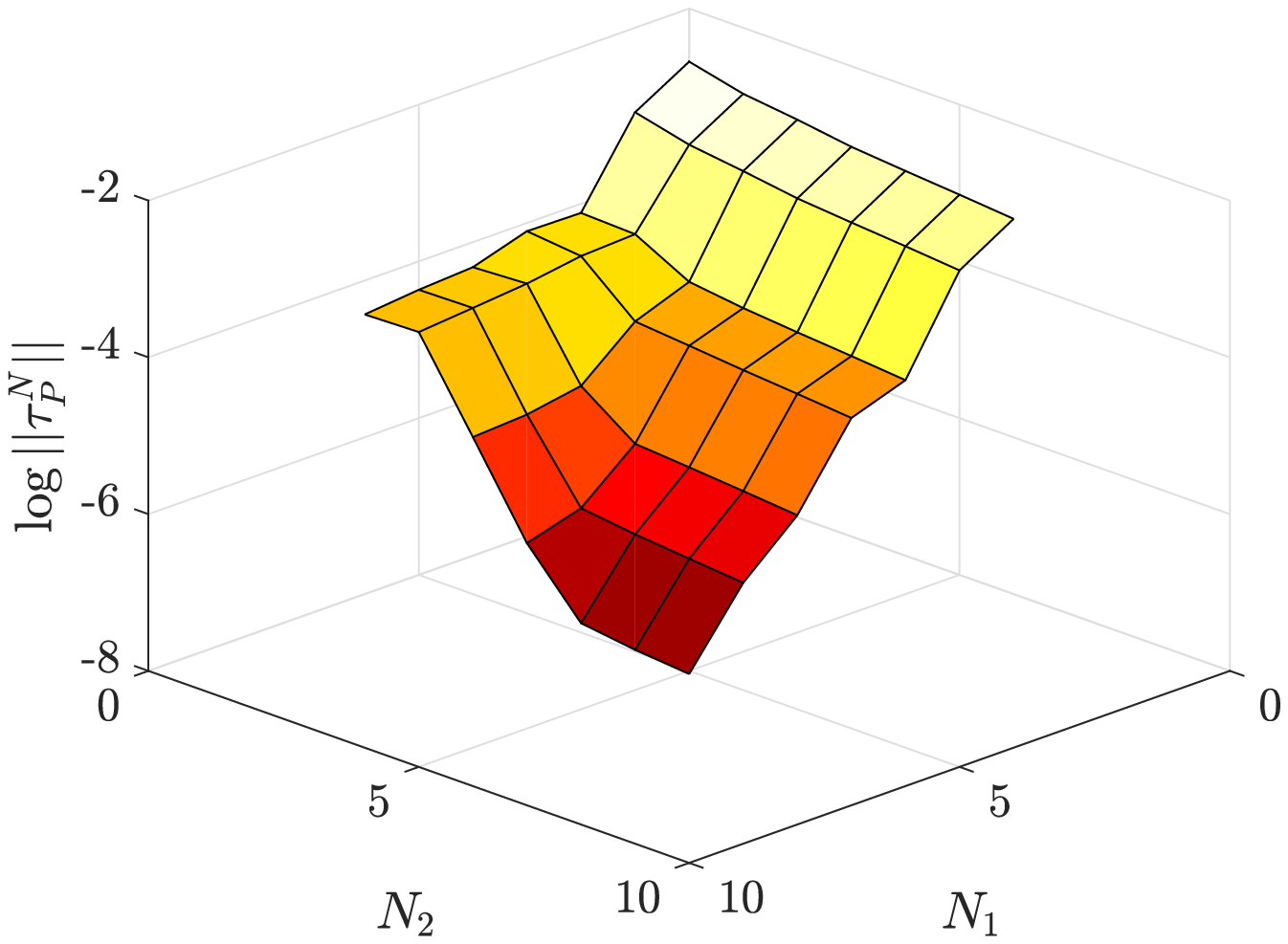}}\qquad
\subfigure[New anisotropic $\tau$-estimation method with high-order extrapolation]{\label{fig:Res_TEmapEstim} \includegraphics[width=0.45\textwidth]{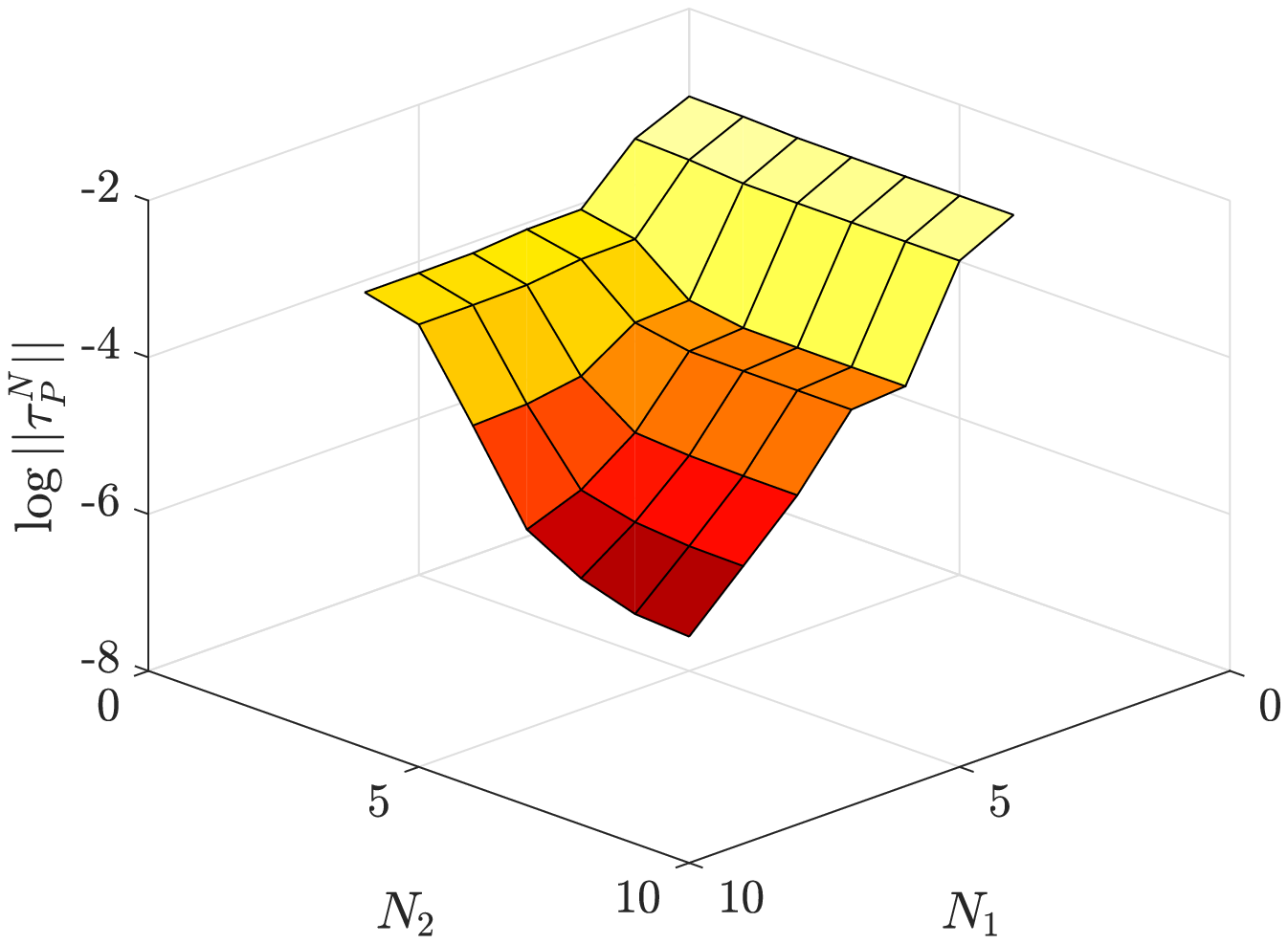}}
\subfigure[Conventional $\tau$-estimation with low order extrapolation (complete hyperplane)]{\label{fig:Res_TEmapKompen} \includegraphics[width=0.45\textwidth]{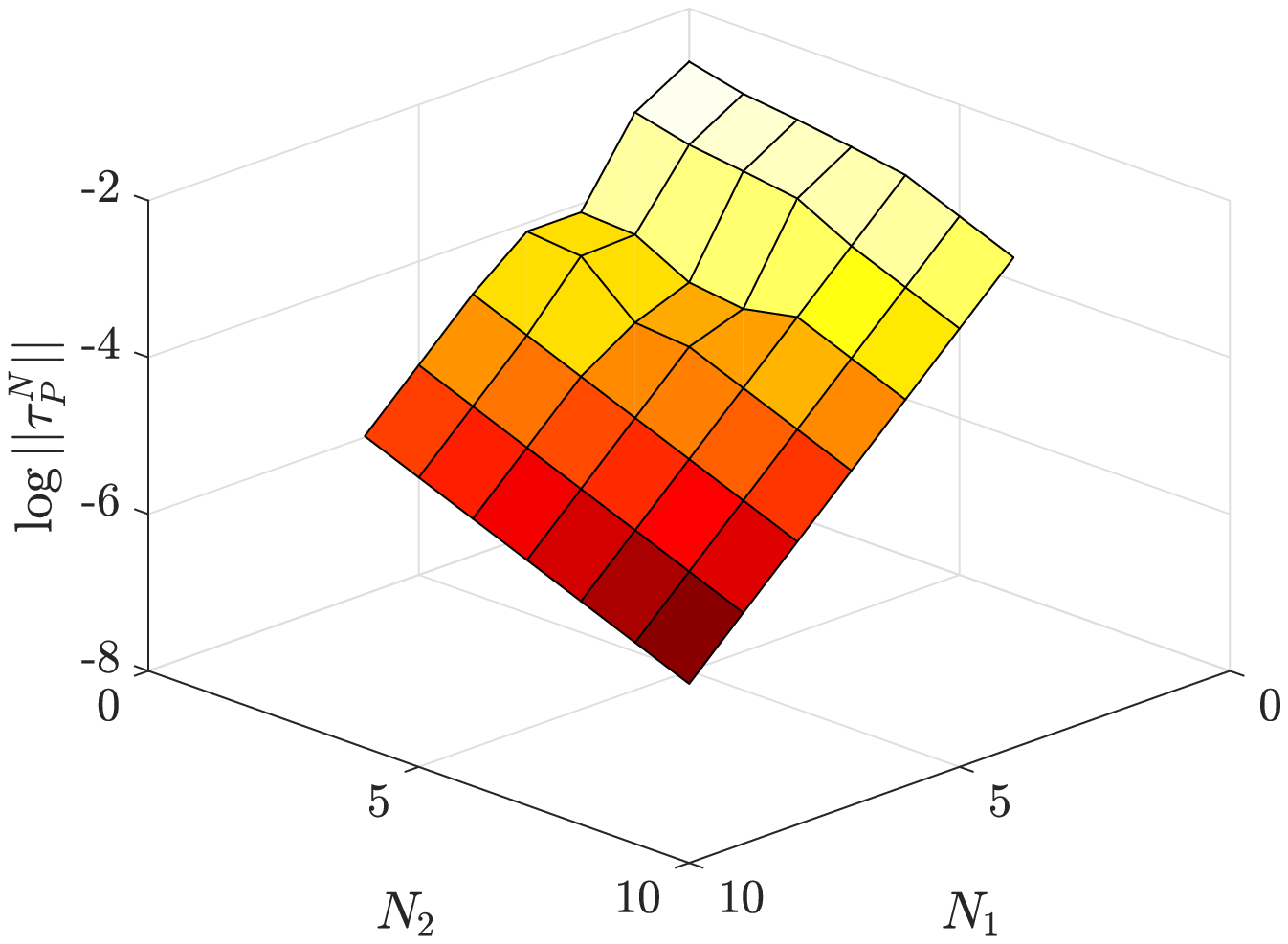}}\qquad
\subfigure[Overlapped surfaces]{\label{fig:Res_TEmapOverlap} \includegraphics[width=0.45\textwidth]{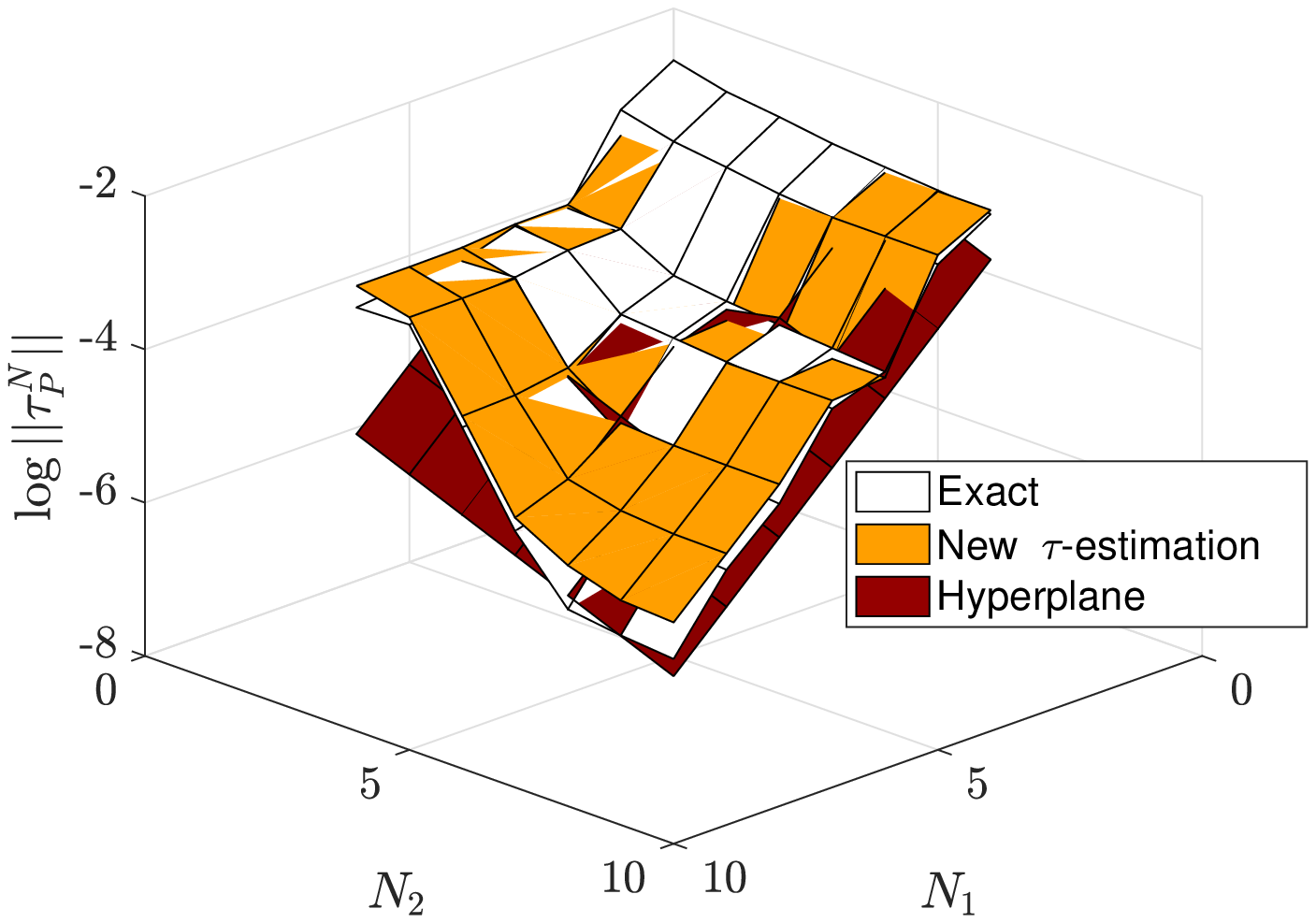}}
\caption{Spatial representation of Two-dimensional anisotropic truncation error maps for the manufactured solutions test case} \label{fig:Res_Surfaces}

\end{center}
\end{figure}

If we generate the truncation error map using the method of Kompenhans et al. \cite{Kompenhans2016} (section \ref{sec:OldPolOrders}), we obtain Figure \ref{fig:TEmapKompen}. Remark that although the method of Kompenhans et al. produces accurate results for $N_i<P$, it fails to predict the behavior of the truncation error for $N_i \ge P$. In fact, using this method the full truncation error map is not being generated, but only the extrapolations for the iso-$N_i$ lines $N_1 = P_1-1$ and $N_2=P_2-1$.\\

\begin{figure}[h!]
\begin{center}
\includegraphics[width=0.4\textwidth]{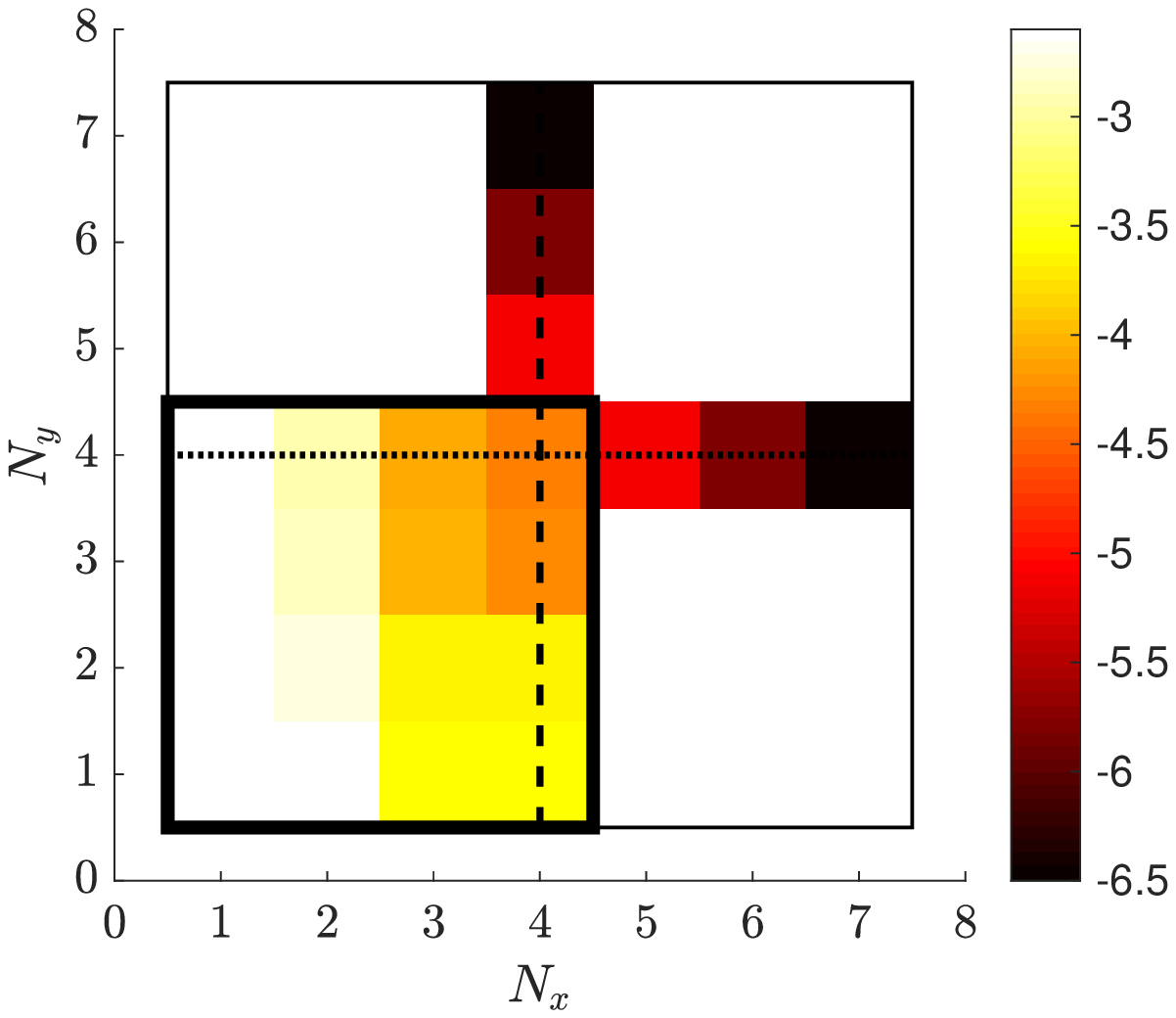}
\caption{Truncation error map estimated using the model of Kompenhans et al. \cite{Kompenhans2016} (logarithmic scale). Outside the black box are the extrapolated values of the estimated truncation error}\label{fig:TEmapKompen}

\vspace*{\floatsep}% https://tex.stackexchange.com/q/26521/5764

\subfigure[Estimated $\tau_1$, $\tau_2$ and $\tau_{5,5}^{Nx,4}$ with $P_1=P_2=5$ (proposed anisotropic model)]{\label{fig:ValidationTauEstimNEWNy4} \includegraphics[width=0.45\textwidth]{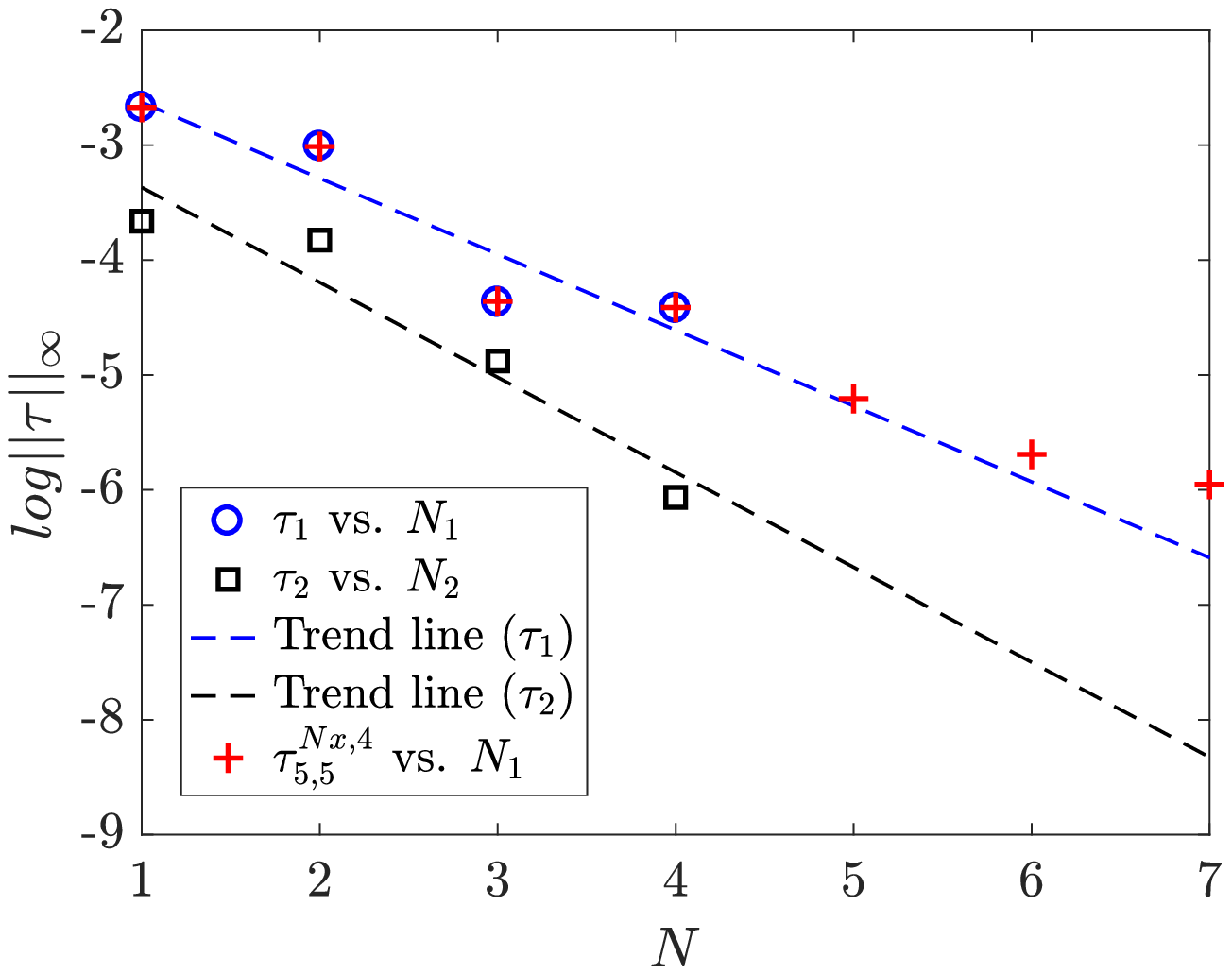}}\qquad
\subfigure[Truncation error vs. $N_1$ for $N_2=4$]{\label{fig:ValidationKompenComparisonNy4} \includegraphics[width=0.45\textwidth]{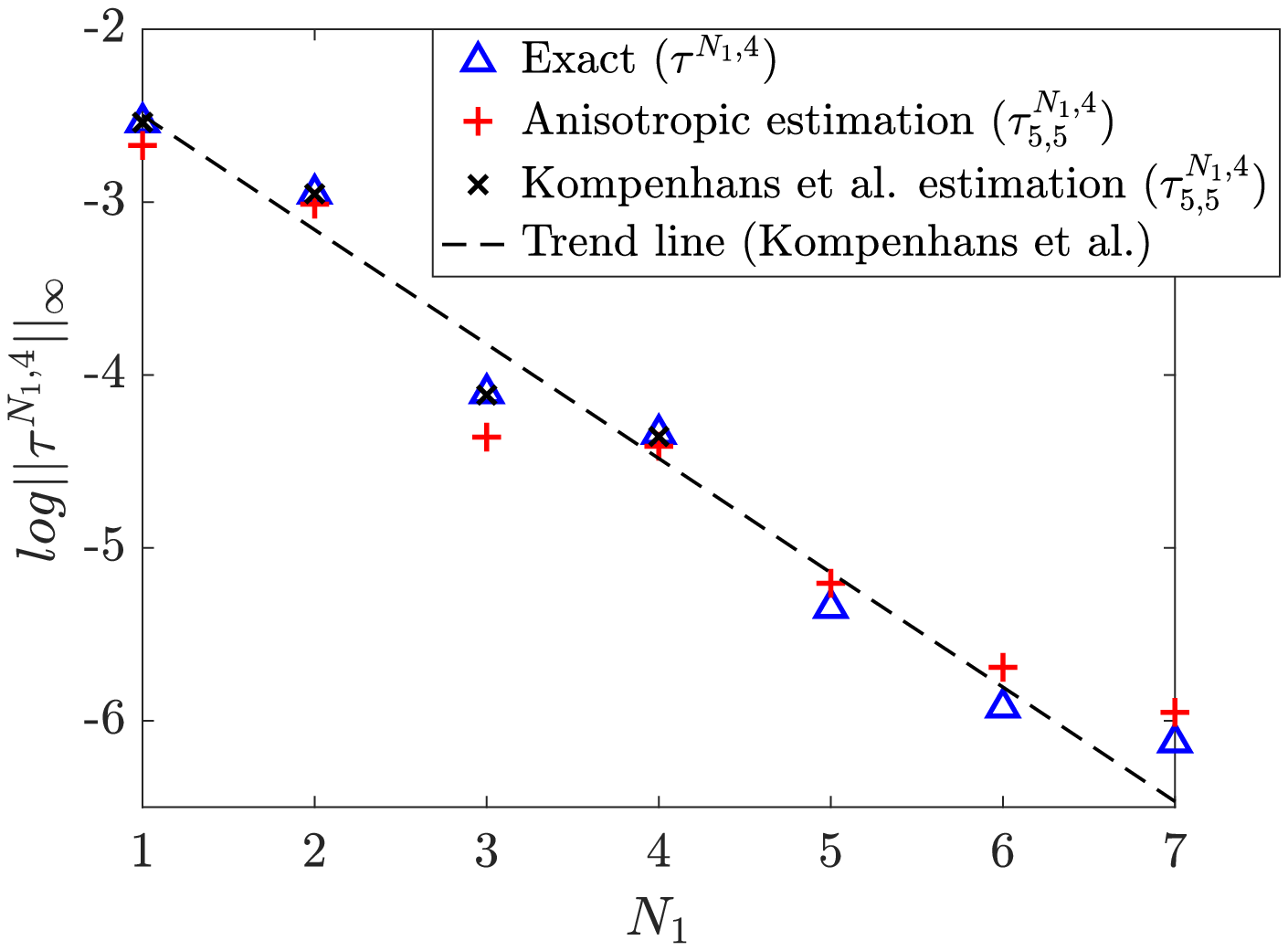}}
\caption{Truncation error estimation using the new anisotropic model (a) and comparison with values obtained using the model of Kompenhans et al. \cite{Kompenhans2016} (b) for $N_2=4$ in element A (dotted line of Figure \ref{fig:TEmapKompen}) - Logarithmic scale} \label{fig:ValidationTauKompenNy4}

\vspace*{\floatsep}% https://tex.stackexchange.com/q/26521/5764

\subfigure[Estimated $\tau_1$, $\tau_2$ and $\tau_{5,5}^{4,Ny}$ with $P_1=P_2=5$ (proposed anisotropic model)]{\label{fig:ValidationTauEstimNEWNx4} \includegraphics[width=0.45\textwidth]{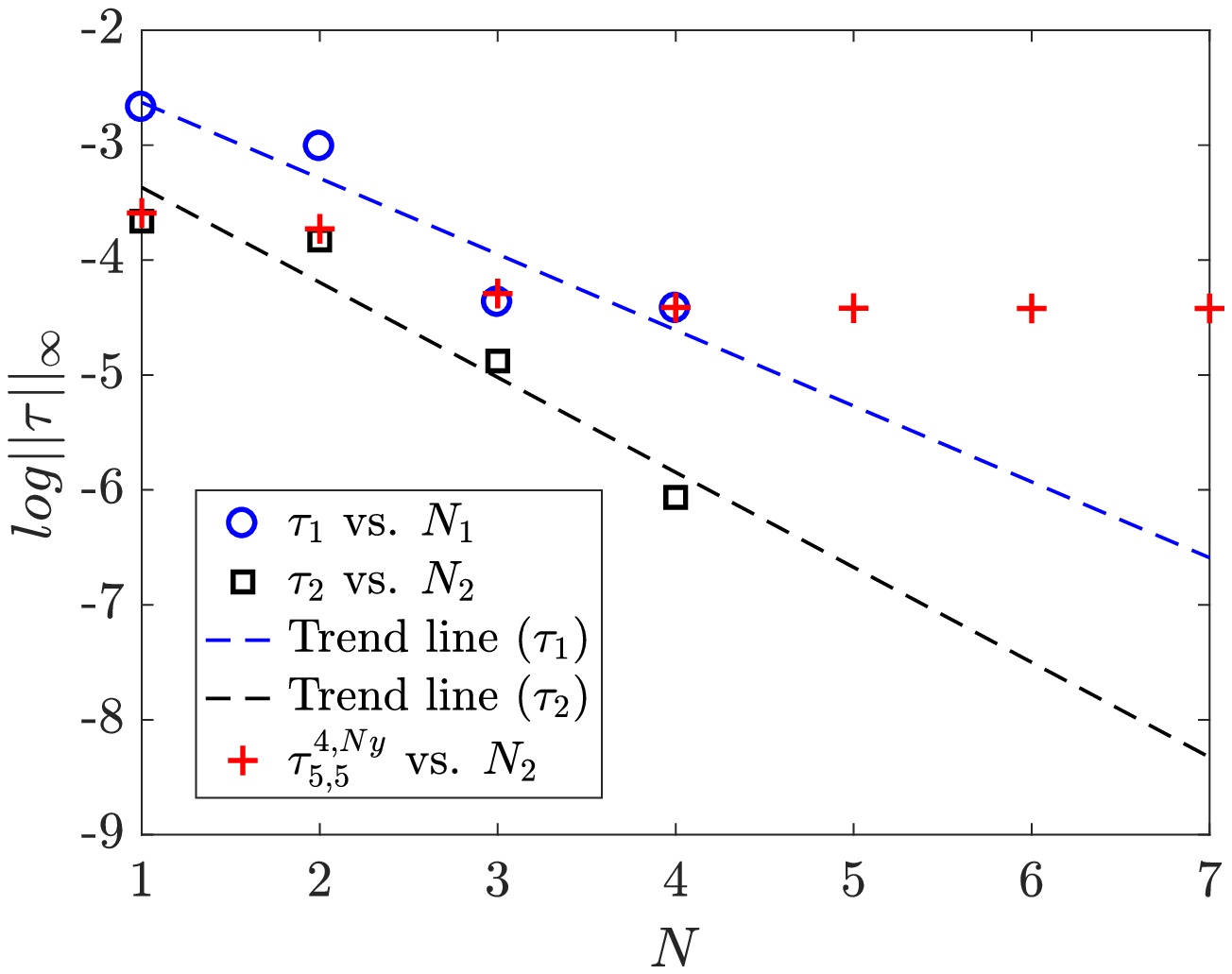}}\qquad
\subfigure[Truncation error vs. $N_2$ for $N_1=4$]{\label{fig:ValidationKompenComparisonNx4} \includegraphics[width=0.45\textwidth]{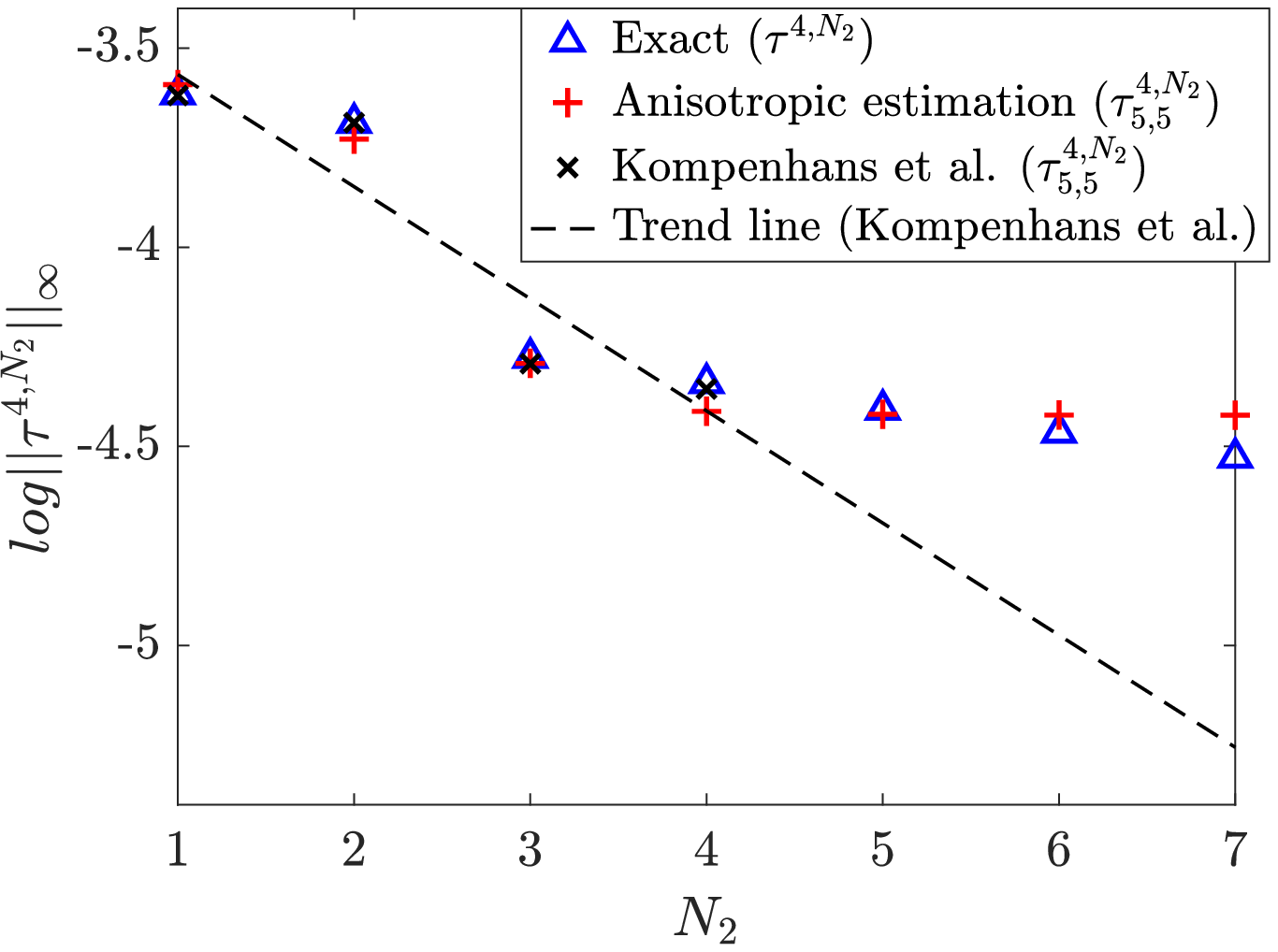}}
\caption{Truncation error estimation using the new anisotropic model (a) and comparison with values obtained using the model of Kompenhans et al. \cite{Kompenhans2016} (b) for $N_1=4$ in element A (dashed line of Figure \ref{fig:TEmapKompen}) - Logarithmic scale} \label{fig:ValidationTauKompenNx4}
\end{center}
\end{figure}

A close inspection of the values of the truncation error for a fixed polynomial order (dashed and dotted lines of Figure \ref{fig:TEmapKompen}) can reveal details about the extrapolated map. Let us first analyze the truncation error for a fixed $N_2=4$. In Figure \ref{fig:ValidationTauEstimNEWNy4} we illustrate how $\tau_{5,5}^{N_1,4}$ is obtained using the new methodology of section \ref{sec:AnisTauEstimation}: the anisotropic contributions of the truncation error, $\tau_1$ and $\tau_2$, are used to generate independent trend lines and their values are then used to compute $\tau_{5,5}^{N_1,4}$. Figure \ref{fig:ValidationKompenComparisonNy4} shows a comparison of this result with the exact truncation error and the one obtained using the method of Kompenhans et al.. It is remarkable that spectral convergence can be observed and both error estimators predict it.\\

Now, let us analyze the case of a fixed $N_1=4$. Figure \ref{fig:ValidationTauEstimNEWNx4} illustrates how $\tau_{5,5}^{4,N_2}$ is obtained. Notice how, in this case, for $N_2 \ge 3$ a stagnation in the decreasing rate of the truncation error occurs because 
\begin{equation}
\norm{\tau_2 (N_2 \ge 3)}_{\infty} \le \norm{\tau_1(N_1=4)}_{\infty}.
\end{equation}

Figure \ref{fig:ValidationKompenComparisonNx4} shows a comparison of this result with the exact truncation error and the one obtained using the method of Kompenhans et al.. Remark that the exact truncation error also exhibits the stagnation behavior for $N_2 \ge 3$, but a linear extrapolation of the values of $\tau$ would under-predict the truncation error for $N_2>4$. The reason is that spectral convergence can be expected for the decoupled terms ($\tau_i$), but not necessarily for the total truncation error along lines of the map (Theorems \ref{theo:AnisTruncErrorConvergence} and \ref{theo:NotSpectralOnLines}). This simple example shows how the anisotropic error estimator formulated in this paper can generate more accurate representations of the truncation error map for $N_i \ge P_i$ than previous estimators.

\subsection{\textit{Non-isolated} truncation error vs. isolated truncation error} \label{sec:ResIsolTruncError}

As was discussed above, both the \textit{non-isolated} and the \textit{isolated} truncation error can be approximated using the anisotropic method introduced in this paper. In this section, we analyze how both estimators perform with the new anisotropic approximation when driving a p-adaptation procedure. The fully converged solution of order $P_1=P_2=5$ is used as the reference mesh for the anisotropic $\tau$-estimation procedure with \textit{high-order extrapolation} explained in section \ref{sec:NewPolOrders}. Different truncation error thresholds are studied in the range $10^{-7} \le \tau_{max} < 10^{-1}$, and the polynomial order is selected after the estimation so that the number of degrees of freedom is minimized (see Figure \ref{fig:ValidationDOFs}). The maximum polynomial order allowed in any direction is selected as $N_{max}=10$, and the minimum polynomial order as $N_{min}=1$.\\

Figure \ref{fig:TauBehavior} shows the \textit{non-isolated} truncation error that was achieved after the mesh adaptation as a function of the specified threshold ($\tau_{max}$), and Figure \ref{fig:IsolTauBehavior} illustrates the \textit{isolated} truncation error that was achieved for different values of $\tau_{max}$. Two plateaux can be observed in both figures, one for $\tau_{max} \le 10^{-5}$ and one for $\tau_{max} \ge 6 \times 10^{-3}$ as a consequence of the limiting polynomial orders. The first plateau corresponds to the minimum $||\tau||_{\infty}$ (and $||\hat \tau||_{\infty}$) that can be achieved when $N_1=N_2=N_{max}=10$, and the second corresponds to the maximum $||\tau||_{\infty}$ (and $||\hat \tau||_{\infty}$) that can be achieved when $N_1=N_2=N_{min}=1$ in every element. For the remaining specified thresholds both estimators perform reasonably well, being the \textit{isolated} truncation error slightly better. The small gap between the ideal and achieved errors is attributed to small errors in the estimation procedure.\\ 
\begin{figure}[htb]
\begin{center}
\subfigure[\textit{Non-isolated} truncation error.]{\label{fig:TauBehavior} \includegraphics[width=0.46\textwidth]{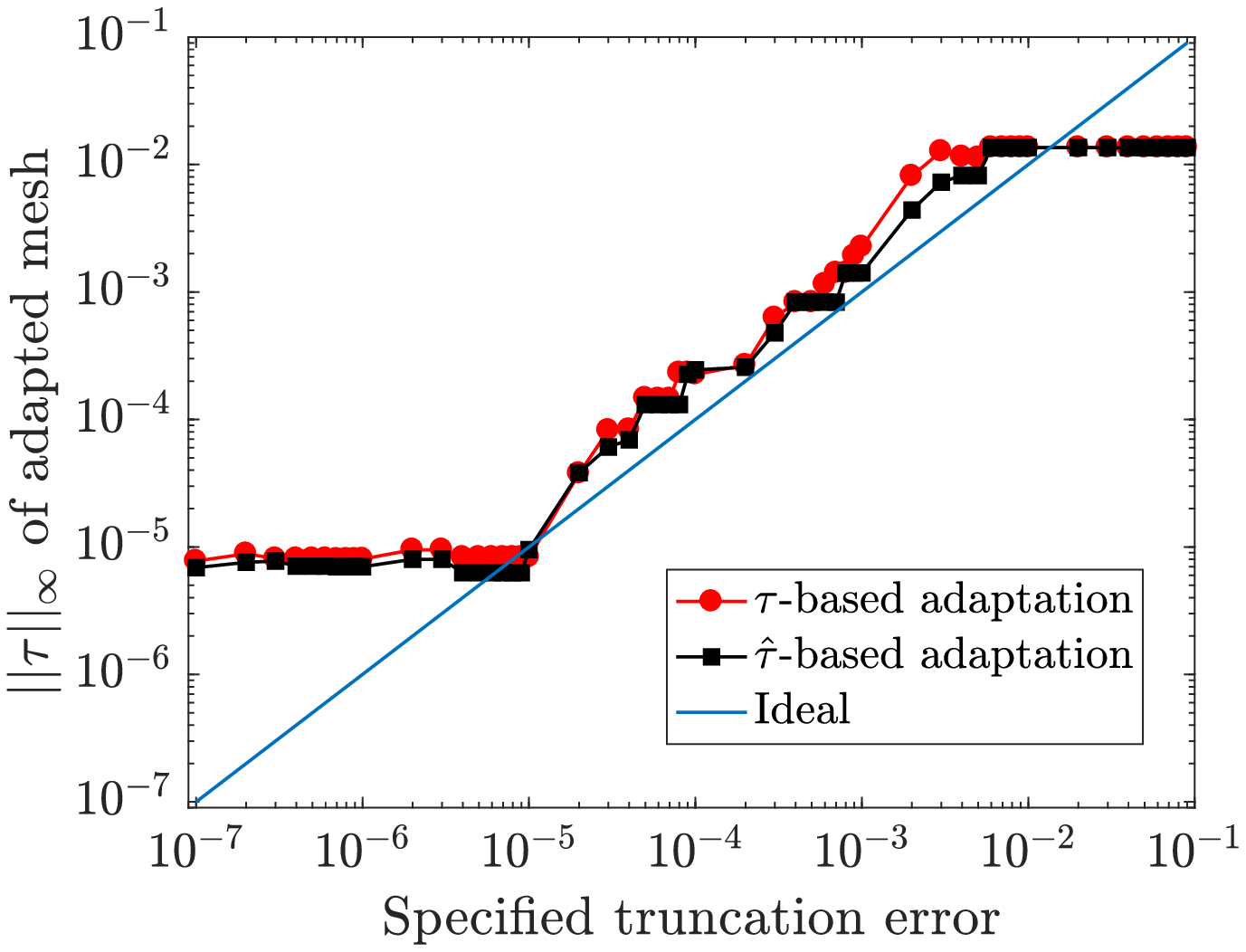}}\qquad
\subfigure[\textit{Isolated} truncation error.]{\label{fig:IsolTauBehavior} \includegraphics[width=0.46\textwidth]{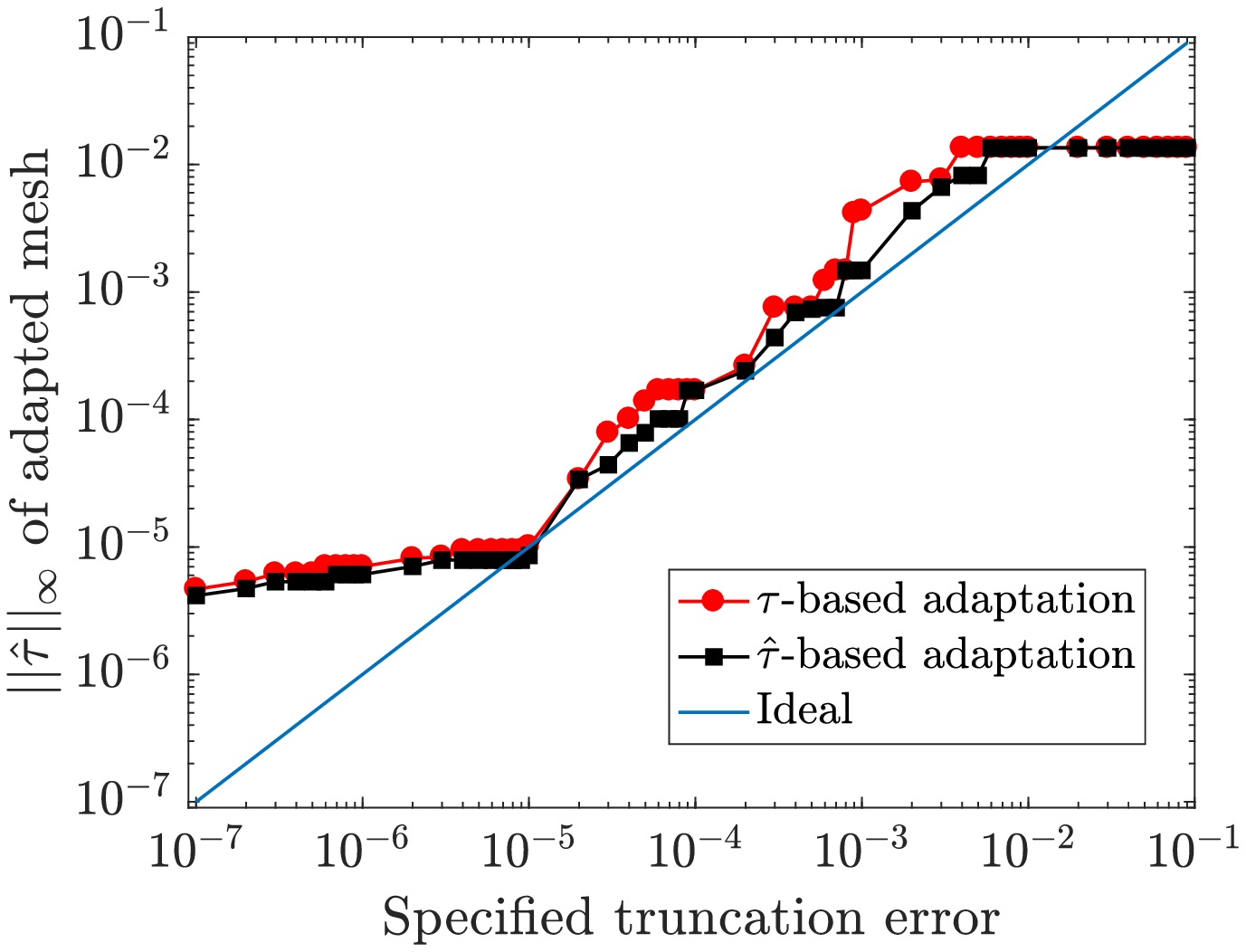}}

\caption{Achieved \textit{non-isolated} truncation error (a) and \textit{isolated} truncation error (b) for adaptation procedures based on the \textit{non-isolated} and the \textit{isolated} truncation error for $N_{max}=10$} \label{fig:ValidationTauBehavior}
\end{center}
\end{figure}

As these results show, controlling the \textit{isolated} truncation error of a mesh also controls its \textit{non-isolated} truncation error: a further advantage of the \textit{isolated} estimator. In fact, we can write the \textit{non-isolated} truncation error in terms of the \textit{isolated} truncation error from the definitions in section \ref{sec:ErrorDefinitions}, and appendixes \ref{sec:IsolTruncError} and \ref{sec:ProofNonIsolated}:
\begin{equation} \label{eq:TauInTermsOfIsolTau}
\tau^N = \hat \tau^N + \int^N_{\partial \Omega} \left( \mathscr{F}(\mathbf{I}^N \bar{\mathbf{q}}) \cdot \mathbf{n} - \mathscr{F}^{*}(\mathbf{I}^N \bar{\mathbf{q}},\mathbf{I}^N \bar{\mathbf{q}}^{\underline{\ }},\mathbf{n}) \right) \phi d \sigma.
\end{equation}

Equation \ref{eq:TauInTermsOfIsolTau} suggests that the \textit{isolated} truncation error is expected to control the \textit{non-isolated} truncation error for sufficiently smooth solutions, for an appropriate choice of the numerical flux. This topic will be addressed in detail in future investigations.\\

Taking into account that the main difference of the \textit{non-isolated} truncation error is that it is affected by neighboring elements, we can conclude that the \textit{isolated} estimator is a better driver for p-adaptation methods than the \textit{non-isolated} truncation error estimator. Namely, because it would be excessively expensive to evaluate every possible combination of polynomial orders for each element of the mesh \textit{and} its neighbors in order to feed the p-adaptation procedure.

\section{Conclusions} \label{sec:Conclusions}
In this paper, we have studied truncation error estimators, their convergence properties and accuracy. The most important conclusions of this work are:
\begin{enumerate}

\item A new technique for evaluating the truncation error was developed which requires less computational resources in the estimation procedure than previous implementations. Furthermore, this technique allows computing extrapolations of the truncation error with enhanced accuracy compared with previous methods. This enables using coarser reference meshes, hence further improving the computational efficiency.

\item The presented method provides truncation error estimations that are accurate enough for performing p-adaptation, as shown in sections \ref{sec:TEmaps} and \ref{sec:ResIsolTruncError}. 

\item According to the analyses conducted in this paper, the \textit{isolated} truncation error is better suited to drive a p-adaptation procedure than its \textit{non-isolated} counterpart. In the first place, because the \textit{non-isolated} error is affected by the discretization in other regions. Second, and as stated in remark \ref{rem:TruncErrorReq}, the \textit{non-isolated} truncation error estimator imposes certain requirements for the extrapolation procedure to work well. This translates into a more expensive $\tau$-estimation. Furthermore, additional requirements are needed in order for the Theorem \ref{theo:NewTauEstimator} to hold with the \textit{non-isolated} truncation error.

\item The method of Kompenhans et al. \cite{Kompenhans2016}, in which every combination of $N = (N_1,N_2,\cdots,N_d)$ is directly evaluated for generating the truncation error map, performs slightly better at estimating the truncation error for $N_i<P_i$ than the proposed error estimator, but fails to predict the truncation error for $N_i \ge P_i$ accurately. A good compromise could be to generate the truncation error map for $N_i<P$ using the method of Kompenhans et al., but then changing to the fully decoupled method for generating the extrapolated map. In this case, however, additional evaluations of the discrete partial differential operator must be performed.

\end{enumerate}

\begin{acknowledgements}
The authors would like to thank David Kopriva for his friendly advise and cooperation. This project has received funding from the European Union’s Horizon 2020 Research and Innovation Program under the Marie Skłodowska-Curie grant agreement No 675008.\\

The authors acknowledge the computer resources and technical assistance provided by the \textit{Centro de Supercomputación y Visualización de Madrid} (CeSViMa).
\end{acknowledgements}

%%%%%%%%%%%%%%%%%%%%%%%%%%%%%%%%%%%%%%%%%%%%%%%%%%%%%%%%%%%%%%%%%%%%%%

%\bibliography{Biblio}   % name your BibTeX data base

%%%%%%%%%%%%%%%%%%%%%%%%%%%%%%%%%%%%%%%%%%%%%%%%%%%%%%%%%%%%%%%%%%%%%%
%% APPENDIX!!

\appendix

\section{Isolated truncation error dependence on inteprolation error} \label{sec:IsolTruncError}

According to definition \ref{Def:IsolTruncError} and equation \ref{eq:IsolatedDiscreteNonlinOperator}, the \textit{isolated} truncation error in the DGSEM can be expressed for any basis function $\phi$ in an element $e$ as

\begin{equation} \label{eq:AppIsolTruncError}
\hat \tau^N \bigr\rvert_{\Omega^e} = \hat{\mathcal{R}} (\mathbf{I}^N \bar{\mathbf{q}}) = \int^N_{\Omega^e} \mathbf{s}^N \phi d\Omega + \int^N_{\Omega^e} \mathscr{F}^N \cdot \nabla \phi d \Omega - \int^N_{\partial \Omega^e} \mathscr{F}^N \cdot \mathbf{n}  \phi d \sigma,
\end{equation}
where the superindex $N$ on the integrals indicates that they are approximated with a Gaussian quadrature of order $N$ and the superindex $e$ has been dropped for readability. Since the DGSEM is a collocation method, the value computed with equation \ref{eq:AppIsolTruncError} corresponds to the \textit{isolated} truncation error on the node of the basis function $\phi$. The terms $\mathbf{s}^N$ and $\mathscr{F}^N$ can be expressed in terms of the interpolation error as

\begin{equation} \label{eq:AppInterpol}
\mathscr{F}^N = \mathbf{I}^N \mathscr{F}(\bar{\mathbf{q}}) = \mathscr{F}(\bar{\mathbf{q}}) - \varepsilon^N_{\mathscr{F}}, \ \ \ \mathbf{s}^N = \mathbf{I}^N \mathbf{s} = \mathbf{s} - \varepsilon^N_{\mathbf{s}}.
\end{equation}

Inserting equation \ref{eq:AppInterpol} into \ref{eq:AppIsolTruncError}, integrating by parts, and expressing everything with $L_2(\Omega)$ inner product notation we obtain,

\begin{equation}
\hat \tau^N \bigr\rvert_{\Omega^e} = - \cancel{\left( \varepsilon_{\mathbf{s}}^N,\phi \right)^N_{\Omega^e} } + \left( \nabla \cdot \varepsilon^N_{\mathscr{F}} , \phi \right)^N_{\Omega^e} + \mathcal{O} \left( e_{\int}^N \right),
\end{equation}

where $(\cdot,\cdot)^N_{\Omega^e}$ stands for the $L_2$ product operator evaluated with a quadrature of order $N$ in the domain ${\Omega^e}$. The first term on the right-hand side vanishes since the value of $\varepsilon_{\mathbf{s}}^N$ is zero on the quadrature nodes (the DGSEM is a collocation method). Furthermore, it is reasonable to neglect the quadrature error since it is of a lower order of magnitude than the value of the integral. Therefore, we obtain

\begin{equation}
\hat \tau^N \bigr\rvert_{\Omega^e} \approx \left( \nabla \cdot \varepsilon^N_{\mathscr{F}} , \phi \right)^N_{\Omega^e}.
\end{equation}

\section{Anisotropic \textit{non-isolated} truncation error estimation} \label{sec:ProofNonIsolated}
In this section, we show briefly that the \textit{non-isolated} truncation error can be estimated anisotropically using theorem \ref{theo:NewTauEstimator}. In order to do so, we need some additional assumptions.

\subsection{\textbf{Additional assumptions}}
As in section \ref{sec:AnisTauEstimation}, following assumptions are a consequence of the tensor product basis functions of the DGSEM and hold for sufficiently smooth solutions in the asymptotic range:
\begin{enumerate}[label=(\alph*)]
\setcounter{enumi}{2}
\item The discretization error has an anisotropic behavior and, therefore, can be  decoupled in directional components. For the 2D case:
	\begin{equation} \label{eq:DiscErrorAnis}
	\epsilon^{N_1N_2} = \epsilon_1^{N_1N_2} + \epsilon_2^{N_1N_2}.
	\end{equation}
	As in (a), $\epsilon_i$ is the projection of the global discretization error, $\epsilon$, into a local direction, $i$.
\item The \textit{locally-generated} discretization error in each direction depends only on the polynomial order in that direction:
	\begin{equation} \label{eq:DiscErrorDependace}
	\epsilon_{\Omega,i}^{N_1N_2} = \epsilon_{\Omega,i}^{N_1N_2}(N_i)
	\end{equation}
\end{enumerate}
Similar as in remark \ref{rem:TruncErrorReq}, and for reasons that will become clear at the end of the proof, following additional assumption is required:
\begin{enumerate}[label=(\alph*)]
\setcounter{enumi}{4}
\item The $\tau$-estimation procedure is performed element-wise while keeping the polynomial order in other elements \textit{sufficiently} high so that:
\begin{equation}
\norm{\epsilon^N_{\partial \Omega}} \ll \norm{\epsilon^N_{\Omega}}
\end{equation}
\end{enumerate}

As (a) and (b), assumptions (c) and (d) also follow from the work of Rubio et al. \cite{Rubio2015,rubio2015truncation}. Remark that assumptions (a), (b), (c) and (d) are consistent with the dependence of the \textit{non-isolated} truncation error on the discretization error (equation \ref{eq:TruncErrorDiscError}). \\ 

Let us note that the assumption (d) implies that, for smooth solutions in the asymptotic range, the discretization error in one direction does not change considerably when the polynomial order in another direction is changed:
\begin{eqnarray}\label{eq:AnisFASAssump}
\epsilon_j^{N_iP_j} \approx \epsilon_j^{P_iP_j},  \nonumber \\
\epsilon_i^{N_iP_j} \neq    \epsilon_i^{P_iP_j},
\end{eqnarray}

with $i \ne j$, and $1 \le i,j \le 2$.\\

\begin{proof}

Following the same procedure as in Appendix \ref{sec:IsolTruncError}, according to definition \ref{Def:TruncError} and equation \ref{eq:Rdiscrete}, the \textit{non-isolated} truncation error in the DGSEM can be expressed for any basis function $\phi$ in an element $e$ as

\begin{equation} \label{eq:AppTruncError}
\tau^N \bigr\rvert_{\Omega^e} = \mathcal{R} (\mathbf{I}^N \bar{\mathbf{q}}) = \int^N_{\Omega^e} \mathbf{s}^N \phi d\Omega + \int^N_{\Omega^e} \mathscr{F}(\mathbf{I}^N\bar{\mathbf{q}}) \cdot \nabla \phi d \Omega - \int^N_{\partial \Omega^e} \mathscr{F}^{*}(\mathbf{I}^N \bar{\mathbf{q}},\mathbf{I}^N \bar{\mathbf{q}}^{\underline{\ }},\mathbf{n}) \phi d \sigma,
\end{equation} 
where $\bar{\mathbf{q}}^{\underline{\ }}$ is the external (neighbor element's) solution and the superindex ``$e$" has been dropped for the local solution. Since the DGSEM is a collocation method, the value computed with equation \ref{eq:AppTruncError} corresponds to the \textit{non-isolated} truncation error on the node of the basis function $\phi$. After inserting the definition of discretization error (def. \ref{Def:DiscError}), $\bar{\mathbf{q}}=\bar{\mathbf{q}}^N+\epsilon^N$, and expanding the fluxes using Taylor series we obtain

\begin{equation}
\tau^N \bigr\rvert_{\Omega^e} \approx
\int^N_{\Omega^e} \frac{\partial \mathscr{F}}{\partial \mathbf{q}} \biggr\rvert_{\bar{\mathbf{q}}^N} \epsilon^N \cdot \nabla \phi d \Omega 
- \int^N_{\partial \Omega^e} \frac{\partial \mathscr{F^*}}{\partial \mathbf{q}} \biggr\rvert_{\bar{\mathbf{q}}^N,\bar{\mathbf{q}}^{\bunderline{N}},\mathbf{n}} \epsilon^N \phi d \sigma 
- \int^N_{\partial \Omega^e} \frac{\partial \mathscr{F^*}}{\partial \mathbf{q}^{\underline{\ }}} \biggr\rvert_{\bar{\mathbf{q}}^N,\bar{\mathbf{q}}^{\bunderline{N}},\mathbf{n}} \epsilon^{\bunderline{N}} \phi d \sigma,
\end{equation}
where the interpolant of the discretization error is omitted for readability ($\mathbf{I}^N\epsilon^N \rightarrow \epsilon^N$), $\epsilon^N$ is the discretization error of the element $e$, and $\epsilon^{\bunderline{N}}$ is the discretization error of a neighbor element connected through the surface $\partial \Omega$. Notice that, for the sake of readability, the symbol for the external polynomial orders is the same as of the internal ones, i.e. $N$, although they can be different.\\

We now want to approximate the \textit{non-isolated} truncation error through $\tau$-estimation. We part from the definition of the discretization error (equation \ref{eq:DiscError}). Adding and subtracting the discrete solution on a higher order grid, $\mathbf{q}^P$, yields
\begin{align*}
\epsilon^N &= \bar{\mathbf{q}}-\bar{\mathbf{q}}^P+\bar{\mathbf{q}}^P-\bar{\mathbf{q}}^N \\
\epsilon^N &= \epsilon^P +\bar{ \mathbf{q}}^P -\bar{ \mathbf{q}}^N.
\end{align*}

Reorganizing we have
\begin{equation} \label{eq:AppFineSol}
\bar{\mathbf{q}}^P = \bar{\mathbf{q}}^N + \epsilon^N - \epsilon^P.
\end{equation}

Therefore, the $\tau$-estimation yields

\begin{multline} \label{eq:AppTruncErrorEstimate}
\tau_P^N \bigr\rvert_{\Omega^e} = \mathcal{R} (\mathbf{I}^N \bar{\mathbf{q}}^P) \approx
\int^N_{\Omega^e} \frac{\partial \mathscr{F}}{\partial \mathbf{q}} \biggr\rvert_{\bar{\mathbf{q}}^N} (\epsilon^N - \epsilon^P) \cdot \nabla \phi d \Omega 
- \int^N_{\partial \Omega^e} \frac{\partial \mathscr{F^*}}{\partial \mathbf{q}} \biggr\rvert_{\bar{\mathbf{q}}^N,\bar{\mathbf{q}}^{\bunderline{N}},\mathbf{n}} (\epsilon^N - \epsilon^P) \phi d \sigma  \\
- \int^N_{\partial \Omega^e} \frac{\partial \mathscr{F^*}}{\partial \mathbf{q}^{\underline{\ }}} \biggr\rvert_{\bar{\mathbf{q}}^N,\bar{\mathbf{q}}^{\bunderline{N}},\mathbf{n}} (\epsilon^{\bunderline{N}} - \epsilon^{\bunderline{P}}) \phi d \sigma.
\end{multline}

Since it is possible to decouple the discretization error inside our analyzed element in a \textit{locally-generated} and an \textit{externally-generated} component (equation \ref{eq:DiscErrorLocalExt}), equation \ref{eq:AppTruncErrorEstimate} can be rewritten as

\begin{align} \label{eq:AppTruncErrorEstimate2}
\tau_P^N \bigr\rvert_{\Omega^e} \approx
\int^N_{\Omega^e} \frac{\partial \mathscr{F}}{\partial \mathbf{q}} \biggr\rvert_{\bar{\mathbf{q}}^N} (\epsilon_{\Omega}^N - \epsilon_{\Omega}^P) \cdot \nabla \phi d \Omega 
- \int^N_{\partial \Omega^e}& \frac{\partial \mathscr{F^*}}{\partial \mathbf{q}} \biggr\rvert_{\bar{\mathbf{q}}^N,\bar{\mathbf{q}}^{\bunderline{N}},\mathbf{n}} (\epsilon_{\Omega}^N - \epsilon_{\Omega}^P) \phi d \sigma \nonumber \\
+ \int^N_{\Omega^e} \frac{\partial \mathscr{F}}{\partial \mathbf{q}} \biggr\rvert_{\bar{\mathbf{q}}^N} (\epsilon_{\partial \Omega}^N - \epsilon_{\partial \Omega}^P) \cdot \nabla \phi d \Omega 
- \int^N_{\partial \Omega^e}& \frac{\partial \mathscr{F^*}}{\partial \mathbf{q}} \biggr\rvert_{\bar{\mathbf{q}}^N,\bar{\mathbf{q}}^{\bunderline{N}},\mathbf{n}} (\epsilon_{\partial \Omega}^N - \epsilon_{\partial \Omega}^P) \phi d \sigma \nonumber \\
&- \int^N_{\partial \Omega^e} \frac{\partial \mathscr{F^*}}{\partial \mathbf{q}^{\underline{\ }}} \biggr\rvert_{\bar{\mathbf{q}}^N,\bar{\mathbf{q}}^{\bunderline{N}},\mathbf{n}} (\epsilon^{\bunderline{N}} - \epsilon^{\bunderline{P}}) \phi d \sigma.
\end{align}

Equation \ref{eq:AppTruncErrorEstimate2} holds even for anisotropic representations, i.e. $N=(N_1,N_2,N_3)$ and $P=(P_1,P_2,P_3)$. Remark that if the polynomial order of the elements that are not being analyzed is maintained as high as in the reference mesh, $\epsilon_{\partial \Omega}^P$ cancels out $\epsilon_{\partial \Omega}^N$ and $\epsilon^{\bunderline{N}} - \epsilon^{\bunderline{P}} \approx 0$. I.e., the $\tau$-estimation provides the \textit{locally-generated} truncation error.\\

Let us now consider the case of 2D anisotropic coarsening in the direction $i$ ($\mathbf{N}=(N_i,P_j)$, $\mathbf{P}=(P_i,P_j)$). Taking into account assumptions $(c)$ and $(d)$, we obtain

\begin{align} \label{eq:AppTruncErrorEstimate3}
\tau_{\mathbf{P}}^{\mathbf{N}} \bigr\rvert_{\Omega^e} \approx
\int^{\mathbf{N}}_{\Omega^e} \frac{\partial \mathscr{F}}{\partial \mathbf{q}} \biggr\rvert_{\bar{\mathbf{q}}^{\mathbf{N}}} (\epsilon_{\Omega,i}^{\mathbf{N}} - \epsilon_{\Omega,i}^{\mathbf{P}}) \cdot \nabla \phi d \Omega 
- \int^{\mathbf{N}}_{\partial \Omega^e}& \frac{\partial \mathscr{F^*}}{\partial \mathbf{q}} \biggr\rvert_{\bar{\mathbf{q}}^{\mathbf{N}},\bar{\mathbf{q}}^{\bunderline{{\mathbf{N}}}},\mathbf{n}} (\epsilon_{\Omega,i}^{\mathbf{N}} - \epsilon_{\Omega,i}^{\mathbf{P}}) \phi d \sigma \nonumber \\
+ \int^{\mathbf{N}}_{\Omega^e} \frac{\partial \mathscr{F}}{\partial \mathbf{q}} \biggr\rvert_{\bar{\mathbf{q}}^{\mathbf{N}}} (\epsilon_{\partial \Omega}^{\mathbf{N}} - \epsilon_{\partial \Omega}^{\mathbf{P}}) \cdot \nabla \phi d \Omega 
- \int^{\mathbf{N}}_{\partial \Omega^e}& \frac{\partial \mathscr{F^*}}{\partial \mathbf{q}} \biggr\rvert_{\bar{\mathbf{q}}^{\mathbf{N}},\bar{\mathbf{q}}^{\bunderline{{\mathbf{N}}}},\mathbf{n}} (\epsilon_{\partial \Omega}^{\mathbf{N}} - \epsilon_{\partial \Omega}^{\mathbf{P}}) \phi d \sigma \nonumber \\
&- \int^{\mathbf{N}}_{\partial \Omega^e} \frac{\partial \mathscr{F^*}}{\partial \mathbf{q}^{\underline{\ }}} \biggr\rvert_{\bar{\mathbf{q}}^{\mathbf{N}},\bar{\mathbf{q}}^{\bunderline{{\mathbf{N}}}},\mathbf{n}} (\epsilon^{\bunderline{{\mathbf{N}}}} - \epsilon^{\bunderline{{\mathbf{P}}}}) \phi d \sigma.
\end{align}

Finally, if assumption $(c)$ and $(e)$ hold, the anisotropic version of equation \ref{eq:AppTruncErrorEstimate2} ($N=(N_1,N_2,N_3)$) can be reconstructed by summing all the directional components (equation \ref{eq:AppTruncErrorEstimate3}) if the quadrature errors are neglected. I.e., we recover equation \ref{eq:NewTauEstimator}:
\begin{equation}
\tau^{N_1N_2} \approx \tau^{N_1P_2}_{P_1P_2} +  \tau^{P_1N_2}_{P_1P_2}
\end{equation}
\end{proof}

\section{The Navier-Stokes equations}\label{sec:NS}

The compressible Navier-Stokes equations in conservative form can be written in non-dimensional form as
\begin{equation} \label{eq:NSeq}
\mathbf{q}_t + \nabla \cdot \left( \mathscr{F}^a - \mathscr{F}^{\nu} \right) = \mathbf{s}, 
\end{equation}

where the conserved variables are $\mathbf{q} = \left( \rho, \rho u, \rho v, \rho w, \rho e  \right)^T$ ($\rho$ is the density; $u$, $v$ and $w$ are the velocity components; and $e$ is the specific total energy), $\mathbf{s}$ is an external source term, and $\mathscr{F}^a$ and $\mathscr{F}^{\nu}$ are called the advective and diffusive flux dyadic tensors, respectively, which depend on $\mathbf{q}$. Expanding the fluxes in Cartesian coordinates leads to the expression,

\begin{equation}
\mathbf{q}_t + \mathbf{f}^a_x + \mathbf{g}^a_y + \mathbf{h}^a_z - \frac{1}{\rm{Re}} \left( \mathbf{f}^{\nu}_x + \mathbf{g}^{\nu}_y + \mathbf{h}^{\nu}_z \right) = \mathbf{s}.
\end{equation}
Here, Re is the Reynolds number. The advective fluxes are then defined as

\begin{equation}
\mathbf{f}^a =
          \begin{bmatrix}
           \rho u      \\           
           p + \rho u^2\\
           \rho u v    \\
           \rho u w    \\
           u (\rho e + p)
          \end{bmatrix},
          		\mathbf{g}^a = 
        				\begin{bmatrix}
				           \rho v      \\           
				           \rho u v    \\
           				   p + \rho v^2\\
				           \rho v w    \\
				           v (\rho e + p)
        				\end{bmatrix},
          						\mathbf{h}^a = 
          								\begin{bmatrix}
								           \rho w      \\           
								           \rho u w    \\
				           				   \rho v w    \\
								           p + \rho w^2\\
								           w (\rho e + p)
			          					\end{bmatrix},
\end{equation}
where the pressure $p$ is computed using the calorically perfect gas approximation. On the other hand, the diffusive fluxes are defined as
\begin{align}
\mathbf{f}^{\nu} &=
          \begin{bmatrix}
           0           \\           
           \tau_{xx}   \\
           \tau_{xy}   \\
           \tau_{xz}   \\
           u \tau_{xx} + v \tau_{xy} + w \tau_{xz} + \frac{\kappa}{(\gamma - 1) \rm{Pr} \rm{M}^2} T_x 
          \end{bmatrix}, \\
          		\mathbf{g}^{\nu} &= 
        				\begin{bmatrix}
				           0           \\           
				           \tau_{yx}   \\
				           \tau_{yy}   \\
				           \tau_{yz}   \\
				           u \tau_{yx} + v \tau_{yy} + w \tau_{yz} + \frac{\kappa}{(\gamma - 1) \rm{Pr} \rm{M}^2} T_y \\
        				\end{bmatrix}, \\
          						\mathbf{h}^{\nu} &= 
          								\begin{bmatrix}
								           0           \\           
								           \tau_{zx}   \\
								           \tau_{zy}   \\
								           \tau_{zz}   \\
								           u \tau_{zx} + v \tau_{zy} + w \tau_{zz} + \frac{\kappa}{(\gamma - 1) \rm{Pr} \rm{M}^2} T_z
			          					\end{bmatrix},
\end{align}
where $T$ is the temperature, $\gamma$ is the heat capacity ratio, and $\kappa$ is the thermal diffusivity. The nondimensional parameters are Pr, the Prandtl number; and M, the Mach number. The stress tensor components are computed using the Stokes hypothesis,
\begin{align}
\tau_{ij} &= \mu \left( \frac{\partial v_i}{\partial x_j} + \frac{\partial v_j}{\partial x_i} \right), i \neq j \\
\tau_{ii} &= 2 \mu \left( \frac{\partial v_i}{\partial x_i} + \nabla \cdot \mathbf{V} \right),
\end{align}
with $\mu$ the fluid's viscosity, and $\mathbf{V}$ the flow velocity. For the simulations in this paper we chose the typical parameters for air: $\rm{Pr} = 0.72$, $\gamma = 1.4$, while $\mu$ and $\kappa$ are calculated using Sutherland's law.

\end{document}